\documentclass{article}
\usepackage[a4paper, total={7in,10in}]{geometry}
\usepackage{amsfonts}
\usepackage{amsthm}
\usepackage{graphicx}
\usepackage{caption}
\usepackage{subcaption}
\usepackage{setspace}
\usepackage{hyperref}
\usepackage{epstopdf}
\usepackage{cleveref} 
\usepackage{wrapfig}
\usepackage{float}
\usepackage{multicol}
\usepackage{amsmath}
\usepackage{subcaption}
\theoremstyle{definition}
\newtheorem{theorem}{Theorem}[section]
\newtheorem{remark}[theorem]{Remark}
\newtheorem{proposition}[theorem]{Proposition}
\newtheorem{definition}[theorem]{Definition}

\newtheorem{assumption}[theorem]{Assumption}

\providecommand{\keywords}[1]{\textbf{\textit{Keywords:}} #1}

\newcommand{\mq}{\mathcal{Q}}
\newcommand{\mqc}{\overline{\mathcal{Q}}}
\newcommand{\mm}{{\mathcal{M}}}
\newcommand{\mt}{\mathbb{T}}

\mathchardef\mhyphen="2D

\newcommand{\mf}{\mathcal{F}}

\newcommand{\intt}{\int_{\mt^3}}
\newcommand{\e}{\epsilon}

\begin{document}
\sloppy
\title{$\Gamma$-convergence of a mean-field model of a chiral doped nematic liquid crystal to the Oseen-Frank description of cholesterics}
\author{Jamie M. Taylor\footnote{Address for correspondance: Basque Center for Applied Mathematics (BCAM), Alameda Mazarredo 14, 48009, Bilbao, Vizcaya, Spain. Email: {\tt jtaylor@bcamath.org}}}
\date{}
\maketitle
\begin{abstract}
Systems of elongated molecules, doped with small amounts of molecules lacking mirror symmetry can form macroscopically twisted cholesteric liquid crystal phases. The aim of this work is to rigorously derive the Oseen-Frank model of cholesterics from a more fundamental model concerned with pairwise molecular interactions. A non-local mean-field model of the two-species nematic host/chiral dopant mixture is proposed, and it is shown that Oseen-Frank's elastic free energy for cholesteric liquid crystals can be obtained in a simultaneously large-domain and dilute-dopant asymptotic regime. By techniques of $\Gamma$-convergence, it is shown that in the asymptotic limit dopant-dopant interactions are negligable, the Frank constants and nematic host order parameter are unperturbed by the presence of dopant, but the mirror asymmetry of the dopant-host interaction leads to a macroscopically twisted ground state. The constant of proportionality between the helical wavenumber and dopant concentration, the {\it helical twisting power} (HTP), can be explicitly found through such an analysis, with a non-linear temperature dependence. Depending on the relative strengths of the host-host and host-dopant interactions, it is shown that HTP may increase or decrease with temperature. 

\end{abstract}
\section{Introduction}
\subsection{Motivation}

It has long been known that systems of asymmetric molecules can form mesophases outside of the classical solid-liquid-gas trichotomy \cite{de1995physics}. The earliest classification of such phases described the nematic, smectic and cholesteric phases \cite{friedel1922etats}. In the nematic phase molecules admit (local) orientational ordering but no positional ordering. In the smectic phase molecules admit orientational and periodic one-dimensional positional ordering, but are translationally disordered within two orthogonal dimensions. Thirdly there is the cholesteric phase, which resembles a nematic on small length scales but forms helical structures over larger length scales. The macroscopic chirality of these helical structures is associated with a corresponding microscopic chirality of the constituent molecules. Early attempts to study cholesteric liquid crystals considered purer systems of chiral molecules, though in more recent years however it has become commonplace to consider multi-species mixtures, where an achiral host system that would naturally form a nematic phase is doped by a small quantity (often as little as 1-2\%) of a chiral dopant. These are typically far more stable and have attracted a wealth of interest. From a practical perspective, cholesterics and complex materials with cholesteric character have found applications due to their tunable optical properties, allowing them to be exploited for many purposes such as displays \cite{kitzerow2009blue}, e-paper \cite{yang1994control}, mechanically-tunable lasing \cite{finkelmann2001tunable}, and smart windows \cite{bao2009smart}. The purpose of this work is to establish a greater understanding of the emergence of cholesteric structures in doped nematics with dilute chiral dopant, within the context of asymptotic behaviour of a mean-field model based on long-range, attractive molecular interactions. 

\subsection{Density functional theory and the mean-field}
For a broad introduction to classical density functional theory the reader is directed to \cite{poniewierski1992density}.

In the simplest case of pairwise interacting particles, we tend to idealise interactions via two components. The first is a short-range repulsive interaction, steric in nature and representative of the non-interpenetrability of ``hard" particles. The second is a long-range, attractive interaction typically representative of dispersion interactions. If we have a system of $N$ particles described by their generalised coordinates $(q_i)_{i=1}^N$, and a pairwise interaction energy $U$, then the mean-field approach is to replace the total interaction energy felt by a single particle with coordinate $q_i$ by an effective mean field to which all particles simultaneously contribute to and feel, by taking the approximation 
\begin{equation}\label{eqMeanField}
\sum\limits_{i\neq j,i=1..N}U(q_i,q_j)\approx \int U(q_i,q)\rho(q)\,dq.
\end{equation}
Here $\rho$ is the one-particle distribution function, so that $\rho(q)$ denotes the number of particles with generalised coordinate $q$. By considering the one-particle distribution function $\rho$ as a parameter to be optimised over by a trial free energy, a standard density functional theory approach is to compete the total interaction energy (arising as a weighted average of \eqref{eqMeanField}) against an entropic energy contribution in the density functional to be minimised,
\begin{equation}\label{eqDFT}
\int k_BT \rho(q)\ln \rho(q)\,dq +\frac{1}{2}\int\int\rho(p)\rho(q)U(p,q)\,dp\,dq. 
\end{equation}

One may impose various ansatzes on such a free energy to provide descriptions of particular phases of interest, notably by prescribing symmetries. For example both the classical Maier-Saupe \cite{maier1959einfache} and Onsager \cite{onsager1949effects} models take an ansatz of spatial inhomogeneity to derive models explaining the isotropic-nematic phase transition. Though mathematically they are superficially near-identical in both the expression for the energy and conclusion, the two models differ in spirit. Onsager is based on short-range steric effects of hard particles, and driven entirely by changes in concentration, while Maier-Saupe is driven by long-range attractive forces, with temperature playing a similar role to the inverse of concentration in Onsager. 

Within this work we will be taking a Maier-Saupe type approach of long-range attractive interactions, driven by temperature. However the ansatz of spatial homogeneity will be replaced with periodicity on an asymptotically large three-dimensional domain. Furthermore we consider a two-species system, though the model is fundamentally an example of \eqref{eqDFT} where the species index is part of the generalised coordinate.

\subsection{Continuum theories of nematics}

Nematic liquid crystals are systems of rod like molecules where the centre of mass of particles is disordered, but their long axis, often described by a unit vector $p\in\mathbb{S}^2$ respecting the head-to-tail symmetry $p\sim -p$, admits local ordering. The Landau-de Gennes model of nematics describes a domain $\Omega\subset\mathbb{R}^3$ of nematic liquid crystal by its Q-tensor, a traceless symmetric matrix that is the normalised second order moment of the local orientation distribution function \cite{de1995physics,mottram2014introduction}. If the orientation of the particles are locally described at a point $x\in\Omega$ by a probability distribution on the sphere, $f(x,\cdot)\in\mathcal{P}(\mathbb{S}^2)$, we define  
\begin{equation}
Q(x)=\int_{\mathbb{S}^2}\left(p\otimes p-\frac{1}{3}I\right)f(x,p)\,dp.
\end{equation}  
Taking $Q$ to be our order parameter and representative of local orientational ordering, we consider a free energy minimisation, with the free energy generally of the form 
\begin{equation}
\int_\Omega W(\nabla Q(x),Q(x))+\psi_B(Q(x))\,dx,
\end{equation}
where $W,\psi_B$ are frame-indifferent functions, $\psi_B$ is the ``bulk" contribution to the energy, and $W$ the ``elastic" contribution, satisfying $W(0,Q)=0$ for all $Q$. For simplicity, $W$ is typically taken to be quadratic in $\nabla Q$, where frame invariance leads to the representation that 
\begin{equation}
W(\nabla Q,Q)=\frac{L_1}{2}Q_{ij,k}Q_{ij,k}+\frac{L_2}{2}Q_{ij,j}Q_{ik,k}+\frac{L_3}{2}Q_{ik,j}Q_{ij,k}+\frac{L_4}{4}Q_{lk}Q_{ij,l}Q_{ij,k},
\end{equation}
where Einstein notation is employed, and indices after commas denote derivatives in the corresponding coordinate direction. Such models may be justified by a phenomenological Landau expansion, or inferred from mean field models as gradient expansions \cite{han2015microscopic}. To the author's knowledge, however, the passage from mean-field models to Landau-de Gennes has yet to be understood by means of a rigorous proof, instead typically employing truncated series expansions.

Much contemporary work has been devoted to understanding the behaviour of the Landau-de Gennes Q-tensor model in the large domain limit, particularly in the ``one-constant" case where $L_2=L_3=L_4=0$ \cite{gartland2015scalings,majumdar2010landau,nguyen2013refined}. In such a regime, the bulk energy becomes highly penalised, requiring that in the limit $Q(x)\in\text{arg min }\psi_B$ pointwise, which is a set of matrices of the form $s_0\left(n\otimes n-\frac{1}{3}I\right)$, where $s_0>0$ is a constant depending on various material parameters, and any $n\in\mathbb{S}^2$. As $n$ is the only degree of freedom, in the case when $Q=s_0\left(n\otimes n-\frac{1}{3}I\right)$ for some $n\in W^{1,2}(\Omega,\mathbb{S}^2)$, the energy reduces to the Oseen-Frank energy \cite{frank1958liquid}, 
\begin{equation}\label{eqOF}
\int_\Omega \frac{K_{11}}{2}(\nabla \cdot n)^2+\frac{K_{22}}{2}(n\cdot \nabla \times n)^2+\frac{K_{33}}{2}|n\times\nabla\times n|^2\,dx. 
\end{equation}
It should be noted that while $Q$ respects $n,-n$ equivalence, a unit vector field does not, and this has the consequence that such an $n$ does not always exist in $W^{1,2}(\Omega,\mathbb{S}^2)$. \cite{ball2011orientability,bedford2016function}.
The constants $K_{11},K_{22},K_{33}$, relating to the penalisation of splay, twist and bend deformations, respectively, are known as the Frank constants. The Frank constants are related to the constants $L_i$ in the Landau-de Gennes free energy. There is also a so-called ``saddle-splay" term depending only on Dirichlet boundary conditions which we will neglect for this work. The Oseen-Frank model too can be justified phenomenologically by arguing that the energy must be frame invariant, respect $n,-n$ symmetry, be minimised at $\nabla n=0$, and be well approximated by a quadratic in $\nabla n$ in the small deformation regime. 

While much work has been done on the rigorous asymptotics of obtaining Oseen-Frank from Landau-de Gennes, there has been recent interest in obtaining Oseen-Frank directly from mean-field models as a large domain limit. This has included asymptotic descriptions of solutions to the Euler-Lagrange equation and minimisers \cite{liu2017oseen}, convergence of gradient flow-type energies \cite{liu2017small} and $\Gamma$-convergence \cite{taylor2018oseen}. We will continue in this approach with this work, but for the case of cholesteric liquid crystals formed of a two-species system.

\subsection{Cholesteric liquid crystals}

Cholesteric liquid crystals differ from nematics in that their ground states are twisted helical structures. The Oseen-Frank description of cholesterics amounts to the introduction of a term linear in $\nabla n$, so that the energy becomes 
\begin{equation}\label{eqOFchiral}
\frac{1}{2}\int K_{11}(\nabla \cdot n(x))^2+K_{22}(n(x)\cdot \nabla \times n(x)+q)^2+K_{33}|n(x)\times \nabla \times n(x)|^2\,dx,
\end{equation}
where $q$ is a pseudo-scalar describing the wavenumber of the ground state. As $q$ is both a pseudo-scalar and a material constant, this implies that the system must have some kind of intrinsic chirality, which results from molecular chirality. The extra term linear in $\nabla n$ has a corresponding analogue in the Landau-de Gennes model given by 
\begin{equation}
\frac{L_5}{2}\epsilon_{lik}Q_{lj}Q_{ij,k}.
\end{equation} 
The ground state of \eqref{eqOFchiral} in the absence of frustrating boundary conditions is given by configurations of the form 
\begin{equation}
n(x)=\cos(qx\cdot e_3)e_1+\sin(qx\cdot e_3)e_2
\end{equation} in a right-handed orthonormal basis $e_1,e_2,e_3$, so $e_1\times e_2=e_3$. While $n$ is $\frac{2\pi}{q}$ periodic in the $e_3$ direction, as $n$ and $-n$ represent the same configuration, the periodicity of this structure is $\frac{\pi}{q}$. In the case of an achiral host with chiral dopant, we expect $q$ to be roughly proportional to the concentration of dopant in dilute systems. The constant of proportionality is called the {\it helical twisting power} (HTP). The HTP is generally readily measurable experimentally, however difficult to predict. HTP may vary heavily between the same host across different dopants and vice versa, and HTP may be either increasing or decreasing as a function of temperature (see e.g. \cite{kuball1997taddols}). There have been numerous attempts to estimate helical pitch from molecular properties from a theoretical standpoint in both single-species \cite{belli2014density,harris1997microscopic,van1976molecular} and doped systems \cite{earl2003predictions,kamberaj2004helical,osipov2001helical}. In this work we obtain an explicit expression for the HTP depending on various material constants, host concentration, and scalar order parameters of the host and dopant. These scalar order parameters can themselves be explicitly derived from temperature, host number density and material constants. The novelty of our approach against previous work is twofold. Firstly, the problem is qualitatively a singular perturbation problem, phrased in the language of $\Gamma$-convergence. Secondly, we use few ansatzes on our solutions, with features such as the uniaxiality of solutions and the dopant being (locally) aligned with the host being proven from the analysis. These twofold components are of course related, as the commonly taken ansatzes are those expected to hold in large domains, so the fact they are proven to hold exactly in our analysis is a consequence of the asymptotic regime considered. The homogenised nature of our limiting model also provides simpler conclusions.

\subsection{Outline of the paper and main results}

\Cref{sectionModelling} is devoted to the derivation of a mean-field model for chiral doped nematics in a periodic domain. In the usual mean-field spirit, we describe the energy as a competition between pairwise interactions and entropy, in a two-species mixture of host and dopant. Initially, the model is described in terms of orientation distribution functions of each species, however by taking an ansatz that the pairwise interaction energy is bilinear in the effective molecular orientation descriptor tensors, through symmetry arguments and an exact moment closure procedure the model is reduced to a simpler, macroscopic energy defined in terms of the Q-tensors of the host and dopant species. Significantly, Host-Dopant and Dopant-Dopant interactions display chirality through antisymmetric terms in the interaction energy. We then consider an asymptotic scaling corresponding to a large domain with dilute dopant. This leads to the energy $\mf_\e$ as defined in \Cref{defEnergy}. The energy admits a small parameter $\epsilon>0$ which is both proportional to the concetration of dopant and inversely proportional to the sample size. This particular scaling is relevant as it is expected that the periodicity of the cholesteric ground state should be $\frac{\pi}{h\rho_D}$ where $h$ is the HTP and $\rho_D$ is the number density of dopant. Thus if we expect to see structures of the same length scale as the domain, we require $\frac{1}{\rho_D}$ to be comparable to the sample size. 

Before proceeding with a rigorous $\Gamma$-convergence argument, we first provide a heuristic argument in \Cref{sectionHeuristic}, motivating the more precise analysis in later sections. By formally replacing non-local finite differences with gradients, we obtain a Landau-de Gennes flavoured free energy in \eqref{eqLandauCholesteric} for the host and dopant Q-tensors. By neglecting high order terms and minimising the free energy at successive orders of the expansion parameter $\epsilon$, we formally obtain the Oseen-Frank free energy for cholesterics, though symmetry arguments are required to reduce the expression into the classical form, and describe the Frank constants in terms of integrals of molecular interactions and the equilibrium scalar order parameters. As the expansion argument is purely heuristic, in \Cref{sectionGammaConvergence} we provide a precise argument in the language of $\Gamma$-convergence for periodic boundary conditions. It is shown that the limiting free energy formally obtained can be described as the $L^2$-Gamma limit of $\mf_\e$ as $\e\to 0$, which gives us the main theorem of this work, \Cref{thmPeriodicGamma}.

Compactness and lower semicontinuity results from \cite{taylor2018oseen} are of use in the proof strategy. The key ingredients are to prove that the Dopant-Dopant interaction terms provide an asymptotically vanishing contribution to the energy, which is a consequence of the dilute regime we consider. The main qualitative difference from the arguments in \cite{taylor2018oseen} are due to the presence of two energy scales within the bulk-type energy, requiring more care to be taken in providing estimates and recovery sequences, and understanding the interactions between the two macroscopic order parameters. 

While the $\Gamma$-convergence argument provides the limiting energy, the representation of the energy is not in the classical form. In \Cref{sectionCoefficients} we employ symmetry arguments to reduce to a more classical representation. The Frank constants for one-component systems were given in terms of molecular interactions in \cite{taylor2018oseen}, which are unaffected in this analysis. It thus remains to obtain an expression for the HTP that results from this analysis. Our analysis provides the following equation, that the HTP satisfies
\begin{equation}
h=\frac{\tilde{\beta}}{\rho_H\tilde{K}_{22}}\frac{s_c}{s_0}.
\end{equation} 
Here $\tilde{K}_{22}$ and $\tilde{\beta}$ are material constants, and can be thought of as the temperature and concentration independent components of the splay Frank constant and the coefficient of $n\cdot \nabla\times n$, given explicitly in terms of integrals of Host-Host and Host-Dopant interactions, respectively. The scalars $s_0$ and $s_c$ are the scalar order parameters of the host and dopant, respectively, and $\rho_H$ is the number density of the host. While they are nonlinearly dependent on many model parameters, explicit expressions for $s_c,s_0$ are available, and the temperature dependence of the HTP is encoded within their ratio. A dimensionless contribution to the HTP, $\tilde{h}=\frac{s_c}{s_0}$, is the only part that is explicitly temperature dependent. It is determined entirely by two dimensionless parameters $\alpha$, representing the relative strength of Host-Host and Host-Dopant bulk interactions, and a rescaled temperature $\tau$. This relationship is numerically tractable, and a plot of $\tilde{h}$ against $\alpha,\tau$ is presented in \Cref{figHTP}. Significantly, for $\alpha>1$, that is when Host-Host interactions are stronger than Host-Dopant, we see HTP increasing as a function of temperature. When $\alpha<1$, HTP becomes a decreasing function of temperature. At the critical value $\alpha=1$, when Host-Host and Host-Dopant interactions are comparable, it is proven that HTP is independent of temperature. 

\section{The model}\label{sectionModelling}

Let $\Omega$ denote a subset of $\mathbb{R}^3$ containing nematic host (H) and a chiral dopant (D). We presume both species to be elongated molecules, and statistically well described by their long axis alone. We describe the system through two functions, $f_H,f_D:\Omega\to\mathcal{P}(\mathbb{S}^2)$. For $x\in\Omega$, $f_X(x,\cdot)$ is interpreted as the probability distribution describing the distribution of long axis orientations of species $X=H,D$ near $x$. The molecules need not necessarily admit perfect cylindrical symmetry, but are assumed to statistically be well described by such a distribution on their long axis. This difference will be essential in the description of the chiral contributions to the energy, and will be expanded upon in \Cref{remarkChirality}.
We assume that the number density of nematic $\rho_H>0$ and dopant $\rho_D>0$ to be constant in space. Denote by $-\tilde{\mathcal{K}}_{XY}(z,p,q)$ the interaction energy between two particles of species $X,Y$ with orientations $p,q$ (respectively), and centres of mass separated by a vector $z\in\mathbb{R}^3$. We then presume the free energy, up to additive constants irrelevant to our analysis, to be of the form 
\begin{equation}\begin{split}\label{eqEnergyMicroscopic}
&k_BT\int_{\Omega\times\mathbb{S}^2}\rho_H f_H(x,p)\ln f_H(x,p)+\rho_D f_D(x,p)\ln f_D(x,p)\,d(x,p)\\
&-\int_{\Omega\times\mathbb{S}^2}\int_{\Omega\times\mathbb{S}^2}\frac{\rho_H^2}{2}\tilde{\mathcal{K}}_{HH}(x-y,p,q)f_H(x,p)f_H(y,q)\,d(y,q)\,d(x,p)\\
&-\int_{\Omega\times\mathbb{S}^2}\int_{\Omega\times\mathbb{S}^2}\rho_H\rho_D\tilde{\mathcal{K}}_{HD}(x-y,p,q)f_H(x,p)f_D(y,q)\,d(y,q)\,d(x,p)\\
&-\int_{\Omega\times\mathbb{S}^2}\int_{\Omega\times\mathbb{S}^2}\frac{\rho_D^2}{2}\tilde{\mathcal{K}}_{DD}(x-y,p,q)f_D(x,p)f_D(y,q)\,d(y,q)\,d(x,p)
\end{split}
\end{equation}
Each line corresponds, respectively, to entropy, Host-Host interactions, Host-Dopant interactions, and Dopant-Dopant interactions.

\subsection{Pairwise interactions}
For each choice of $X,Y$, $\tilde{\mathcal{K}}_{XY}$ must be frame indifferent, so that for all rotations $R\in\text{SO}(3)$
\begin{equation}
\tilde{\mathcal{K}}_{XY}(Rz,Rp,Rq)=\tilde{\mathcal{K}}_{XY}(z,p,q).
\end{equation}
Furthermore, the head-to-tail symmetry requirement gives
\begin{equation}
\tilde{\mathcal{K}}_{XY}(z,p,q)=\tilde{\mathcal{K}}_{XY}(z,\pm p,\pm q).
\end{equation}
Inversion symmetry (or lack thereof) can be characterised by the difference 
\begin{equation}
\tilde{\mathcal{K}}^c_{XY}(z,p,q)=\frac{1}{2}\tilde{\mathcal{K}}_{XY}(z,p,q)-\frac{1}{2}\tilde{\mathcal{K}}_{XY}(-z,-p,-q)=\frac{1}{2}\tilde{\mathcal{K}}_{XY}(z,p,q)-\frac{1}{2}\tilde{\mathcal{K}}_{XY}(-z,p,q).
\end{equation}
$\tilde{\mathcal{K}}^c_{XY}\neq 0$ implies a lack of inversion symmetry, and thus at least one of the molecules must be chiral. We also consider the symmetric part of the interaction, 
\begin{equation}
\tilde{\mathcal{K}}^a_{XY}(z,p,q)=\frac{1}{2}\tilde{\mathcal{K}}_{XY}(z,p,q)+\frac{1}{2}\tilde{\mathcal{K}}_{XY}(-z,p,q),
\end{equation}
so that $\tilde{\mathcal{K}}_{XY}=\tilde{\mathcal{K}}_{XY}^a+\tilde{\mathcal{K}}_{XY}^c$. If $X,Y$ are both representative of achiral molecules, we expect $\tilde{\mathcal{K}}_{XY}=\tilde{\mathcal{K}}_{XY}^a$.

We define the microscopic order parameter for $p\in\mathbb{S}^2$ as the traceless symmetric matrix $\sigma(p)=p\otimes p-\frac{1}{3}I$. For simplicity, evoking to the London dispersion forces \cite{london1930theorie}, we presume that $\tilde{\mathcal{K}}^a_{XY}(z,p,q)$ is bilinear in $\sigma(p),\sigma(q)$. That is, we assume that for every $z\in\mathbb{R}^3$, we have a linear operator $\tilde{K}_{XY}(z):\text{Sym}_0(3)\to\text{Sym}_0(3)$ so that
\begin{equation}
\tilde{\mathcal{K}}_{XY}(z,p,q)=\tilde{K}_{XY}(z)\sigma(p)\cdot \sigma(q). 
\end{equation}
Due to frame indifference, we can express $\tilde{K}_{XY}$ in full generality as 
\begin{equation}\begin{split}
\tilde{K}_X(z)A=&k_1(z)A+k_2(z)\left(Az\otimes z+z\otimes Az\right)+k_3(z)(Az\cdot z)z\otimes z- \frac{Az\cdot z}{3}\left(2k_2(z)+k_3(z)|z|^2\right)I 
\end{split}
\end{equation}
for isotropic scalar functions $k_i(z)$. The term multiplying the identity matrix is to ensure $\text{Tr}\tilde{K}_X(z)A=0$, but will not effect the value of the energy as $A\cdot I=0$ for all traceless matrices $A$. For much of the analysis the frame invariance of $\tilde{\mathcal{K}}_X$ will not be pertinent, but instead we will rely more strongly on upper and lower bounds of the bilinear form. 

As the host is achiral, we assume $\tilde{\mathcal{K}}_{HH}^c(z,p,q)=0$. However the Host-Dopant and Dopant-Dopant interactions must contain some chiral contribution, as the dopant is chiral. For simplicity, we take $\tilde{\mathcal{K}}_{XY}^c(z,p,q)$ to depend linearly on the microscopic order parameters $\sigma(p),\sigma(q)$ also, so that 
\begin{equation}
\begin{split}
\tilde{\mathcal{K}}_{HD}^c(z,p,q)=&\tilde{K}_{cH}(z)\sigma(p)\cdot \sigma(q),\\
\tilde{\mathcal{K}}_{DD}^c(z,p,q)=&\tilde{K}_{cD}(z)\sigma(p)\cdot \sigma(q),
\end{split}
\end{equation}
for some tensors $\tilde{K}_{cH}(z),\tilde{K}_{cD}(z)$, which must have odd symmetry. Viewed as a linear operator on the traceless symmetric matrices, under symmetry considerations we are permitted in the most general case to consider 
\begin{equation}\label{eqChiralInteraction}
\tilde{K}_{cX}(z)A=k_1(z)(AW-WA)-k_2(z)(WAW^2-W^2AW)
\end{equation}
for isotropic functions $k_i(z)$, where $W$ is the skew-symmetric tensor defined by $Wx=z\times x $ for all $x \in \mathbb{R}^3$. Equivalently, $W_{ij}=-\epsilon_{ij\alpha}z_\alpha$ with $\epsilon_{ijk}$ the Levi-Civita tensor. This includes the form considered in \cite{van1976molecular} as a particular case. For a justification of \eqref{eqChiralInteraction} see \Cref{appSymmetry}.

\begin{remark}\label{remarkChirality}
We note that for molecules to be chiral, they necessarily cannot have cylindrical symmetry, raising questions of the validity of such an expression for a chiral interaction described only by the uniaxial microscopic order parameters $\sigma(p)$. This is however justified in \cite{van1976molecular} for one-component systems and exploited in \cite{osipov2001helical} for two-component systems, by considering low-symmetry chiral molecules with a well-defined long axis. It is assumed that all orientations with the long axis pointing in the same direction are equally probable. We could describe this by saying that if the long-axis of the molecule in a reference frame is $e_1$, and $\rho\in\mathcal{P}(\text{SO}(3))$ is the (local) distribution of orientations of the molecule with respect to this reference frame, then $\rho(R)=\tilde\rho(Re_1)$ for some $\tilde\rho\in\mathcal{P}(\mathbb{S}^2)$. Averaging electrostatic interactions between systems of such molecules can obtain components in an interaction energy such as \eqref{eqChiralMeer}, where chiral contributions to the energy remain, despite being described by an order parameter $\sigma(p)$ that cannot encode chirality.
\end{remark}

\subsection{Entropy and order parameters}\label{subsecEntropy}
In this subsection, including \Cref{propSingPotential}, we freely quote results from \cite{ball2010nematic,mottram2014introduction,taylor2016maximum}.

Let $f\in\mathcal{P}(\mathbb{S}^2)$. We define the Q-tensor of $f$ to be 
\begin{equation}\label{eqQtensor}
Q=\int_{\mathbb{S}^2}f(p)\sigma(p)\,dp=\int_{\mathbb{S}^2}f(p)\left(p\otimes p-\frac{1}{3}I\right)\,dp.
\end{equation}
$Q$ is thus a traceless, symmetric, $3\times 3$ matrix. Furthermore, the constraint that $f$ be an $L^1$ probability distribution forces the constraint that the smallest eigenvalue of $Q$, $\lambda_{\min}(Q)$, be strictly greater than $-\frac{1}{3}$. The tracelessness conditions means a given Q-tensor can be one of three flavours:
\begin{itemize}
\item $Q=0$, in which case we say $Q$ is {\it isotropic}, and representative of a disordered system. 
\item $Q$ has exactly two distinct eigenvalues, in which case we say $Q$ is {\it uniaxial}, and representative of a system of molecules roughly aligned with axial symmetry. We can decompose $Q=s\left(n\otimes n-\frac{1}{3}\right)$ for a scalar order parameter $s\neq 0$, and Oseen-Frank's director $n\in\mathbb{S}^2$. If $s>0$, molecules are roughly aligned along $\pm n$, while if $s<0$, molecules typically lie in the plane orthogonal to $n$.
\item $Q$ has three distinct eigenvalues, in which case $Q$ is {\it biaxial}, and representative of an ordered system with lower symmetry than a uniaxial system.
\end{itemize}

Typically we expect to find nematic liquid crystals in a uniaxial state, with positive scalar order parameter \cite{fatkullin2005critical,mottram2014introduction,vollmer2017critical}. The order parameter $s$ can be taken as representative of the degree of ordering of the system, with $s=1$ corresponding to a perfectly ordered state. 

The eigenvalue constraint that $\lambda_{\min}(Q)>-\frac{1}{3}$ means that not all traceless symmetric matrices are physically meaningful. In fact, given a traceless symmetric matrix $Q$, there exists some $L^1$ distribution $f$ with finite Shannon entropy so that \eqref{eqQtensor} holds if and only if $\lambda_{\min}(Q)>-\frac{1}{3}$. Thus we define the set of {\it physical} Q-tensors to be 
\begin{equation}
\mathcal{Q}=\left\{Q\in\mathbb{R}^{3\times 3}:Q=Q^T,\, \text{Tr}(Q)=0,\, \lambda_{\min}(Q)>-\frac{1}{3}\right\}.
\end{equation}
This is an open bounded convex set, and its closure $\bar{\mathcal{Q}}$ simply consists of traceless symmetric matrices with $\lambda_{\min}(Q)\geq -\frac{1}{3}$. We can define a macroscopic analogue of entropy, $\psi_s:\mathcal{Q}\to\mathbb{R}$ by 
\begin{equation}\label{eqSingPotential}
\psi_s(Q)=\min\limits_{f \in \mathcal{A}_Q}\int_{\mathbb{S}^2}f(p)\ln f(p)\,dp,
\end{equation}
where the admissible set $\mathcal{A}_Q$ is defined as
\begin{equation}
\mathcal{A}_Q=\left\{f \in \mathcal{P}(\mathbb{S}^2):\int_{\mathbb{S}^2}f(p)\sigma(p)\,dp=Q\right\}.
\end{equation}
Heuristically, $\psi_s$ is a macroscopic analogue of the entropic contribution to the energy, obtained by a maximum entropy assumption. It satisfies the following properties 
\begin{proposition}\label{propSingPotential}
Let $\psi_s:\mq\to\mathbb{R}$ be defined as in \eqref{eqSingPotential}. Then the following hold. 
\begin{enumerate}
\item $\psi_s$ is strongly convex, $C^\infty$, frame indifferent, and blows up to $+\infty$ as $\lambda_{\min}(Q)\to-\frac{1}{3}$.
\item The derivative of $\psi_s$, denoted $\Lambda:\mathcal{Q}\to\text{Sym}_0(3)$, is a frame indifferent bijection.
\item For every $Q\in\mathcal{Q}$ there exists a unique minimiser $f^Q$ of the minimisation problem in \eqref{eqSingPotential} given by 
\begin{equation}
f^Q(p)=\frac{1}{Z_Q}\exp\left(\Lambda(Q)p\cdot p\right),
\end{equation}
where $Z_Q>0$ is a normalisation constant. 
\item $\Lambda^{-1}:\text{Sym}_0(3)\to\mathcal{Q}$ satisfies 
\begin{equation}\label{eqLambdaMinus1}
\Lambda^{-1}(A)=\left(\int_{\mathbb{S}^2}\exp(Ap\cdot p)\,dp\right)^{-1}\int_{\mathbb{S}^2}\sigma(p)\exp(Ap\cdot p)\,dp.
\end{equation}
\item $\psi_s$ can be written in terms of $\Lambda$ as $\psi_s(Q)=\Lambda(Q)\cdot Q -\ln\int_{\mathbb{S}^2}\exp(\Lambda(Q)p\cdot p)\,dp$.
\item Let $k\in\mathbb{R}$, and define $\psi_B$ by $\psi_B(Q)=\psi_s(Q)-\frac{k}{2}|Q|^2$. Then there exists some $k^*$ so that if $k<k^*$, $\psi_B$ is minimised only at $Q=0$, while for $k>k^*$, $\psi_B$ is minimised at all Q-tensors of the form $s(k)\left(n\otimes n-\frac{1}{3}I\right)$ for any $n\in \mathbb{S}^2$, and some $s:(k^*,\infty)\to \mathbb{R}$. 
\item Let $k>k^*$, $e_1,e_2,e_3$ be an orthonormal basis of $\mathbb{R}^3$, and $Q=s(k)\sigma(e_1)$. Then there exists a constant $c_k>0$ so that if $\xi = s\sigma(e_1)+t\left(e_2\otimes e_2-e_3\otimes e_3\right)$, then $\nabla^2\psi_B(Q)\xi\cdot \xi \geq c_k(s^2+t^2)$ for $s,t$ sufficiently small \cite{li2014local}.
\end{enumerate}
\end{proposition}

\subsection{The total energy - macroscopic form}
Recalling the energy \eqref{eqEnergyMicroscopic}, we now simplify it to a macroscopic form via an exact moment closure. We proceed with only a formal argument, though the arguments can be made precise through the same methodology as in \cite{liu2017oseen}. 

First, we define the macroscopic Q-tensors $Q,\xi:\Omega\to\mqc$ corresponding to the orientations of Host and Dopant molecules (respectively), by 
\begin{equation}
\begin{split}
Q(x)=&\int_{\mathbb{S}^2}\sigma(p)f_H(x,p)\,dp,\\
\xi(x)=&\int_{\mathbb{S}^2}\sigma(p)f_D(x,p)\,dp.\\
\end{split}
\end{equation}
As these are integrals over zero-measure sets of $L^1$ functions, to make these definitions precise one may employ a duality argument as in \cite{liu2017oseen}. We now reduce the interaction terms to integrals involving only the moments. The argument follows for all terms, so we provide the calculation for only the Host-Dopant interaction. Using these definitions, we may simplify the pairwise interaction terms as follows
\begin{equation}\begin{split}
&\int_{\Omega\times\mathbb{S}^2}\int_{\Omega\times\mathbb{S}^2} \tilde{\mathcal{K}}_{HD}(x-y,p,q)f_H(x,p)f_D(y,q)\,d(x,p)\,d(y,q)\\
=&\int_{\Omega\times\mathbb{S}^2}\int_{\Omega\times\mathbb{S}^2} \left(\tilde{K}_{HD}(x-y)+\tilde{K}_{cH}(x-y)\right)\cdot \big(\sigma(p)\otimes\sigma(q)\big)f_H(x,p)f_D(y,q)\,d(x,p)\,d(y,q)\\
=&\int_{\Omega}\int_{\Omega} \left(\tilde{K}_{HD}(x-y)+\tilde{K}_{cH}(x-y)\right)\cdot \left(\int_{\mathbb{S}^2}\sigma(p)f_H(x,p)\,dp\otimes\int_{\mathbb{S}^2}\sigma(q)f_D(y,q)\,dq\right)\,dx\,dy\\
=& \int_{\Omega}\int_{\Omega} \left(\tilde{K}_{HD}(x-y)+\tilde{K}_{cH}(x-y)\right)\cdot \big(Q(x)\otimes\xi(y)\big)\,dx\,dy\\
=&\int_{\Omega}\int_{\Omega} \left(\tilde{K}_{HD}(x-y)Q(x)\cdot \xi(y)+\tilde{K}_{cH}(x-y)Q(x)\cdot \xi(y)\right)\,dx\,dy.
\end{split}
\end{equation}
Analogous arguments give 
\begin{equation}\begin{split}
&\int_{\Omega\times\mathbb{S}^2}\int_{\Omega\times\mathbb{S}^2} \tilde{\mathcal{K}}_{HH}(x-y,p,q)f_H(x,p)f_H(y,q)\,d(x,p)\,d(y,q)\\
=&\int_{\Omega}\int_{\Omega} \tilde{K}_{HH}(x-y)Q(x)\cdot Q(y)\,dx\,dy,\\
&\int_{\Omega\times\mathbb{S}^2}\int_{\Omega\times\mathbb{S}^2} \tilde{\mathcal{K}}_{DD}(x-y,p,q)f_D(x,p)f_D(y,q)\,d(x,p)\,d(y,q)\\
=&\int_{\Omega}\int_{\Omega} \tilde{K}_{DD}(x-y)\xi(x)\cdot \xi(y)+\tilde{K}_{cD}(x-y)\xi(x)\cdot \xi(y)\,dx\,dy.
\end{split}
\end{equation}

We now turn to the entropic contribution to the energy, and only consider the Host term, with the Dopant following by the same reasoning. The definition of $\psi_s$ gives the estimate
\begin{equation}
\int_{\Omega\times\mathbb{S}^2}f_H(x,p)\ln f_H(x,p)\,d(x,p)\geq \int_{\Omega}\psi_s(Q(x))\,dx. 
\end{equation}
This inequality is attained if and only if $f_H$ satisfies for almost every $x\in\Omega$,
\begin{equation}\label{eqfH}
f_H(x,p)=f^{Q(x)}(p)
\end{equation}
where $f^Q$ is as given in \Cref{subsecEntropy} for almost every $x\in\Omega$. Similarly, we can estimate 
\begin{equation}
\int_{\Omega\times\mathbb{S}^2}f_D(x,p)\ln f_D(x,p)\,d(x,p)\geq \int_\Omega \psi_s(\xi(x))\,dx,
\end{equation}
which is attained if and only if and $f_D$ and $\xi$ are related as 
\begin{equation}\label{eqfD}
f_D(x,p)=f^{\xi(x)}(p)
\end{equation}
almost everywhere. In particular, these attained lower bounds mean that minimisers $f_H,f_D$ must be of the forms given in \eqref{eqfH},\eqref{eqfD}, respectively, depending on their Q-tensors. Thus for the sake of minimisation it then suffices to replace the entropic term by the simpler, macroscopic analogue. This leads us to the equivalent macroscopic free energy, 
\begin{equation}\begin{split}\label{eqEnergyMacroscopic}
&\int_{\Omega}k_BT\rho_H\psi_s(Q(x))+k_BT\rho_D\psi_s(\xi(x))\,dx-\frac{\rho_D^2}{2}\int_{\Omega}\int_{\Omega} \tilde{K}_{DD}(x-y)\xi(x)\cdot \xi(y)+\tilde{K}_{cD}(x-y)\xi(x)\cdot \xi(y)\,dx\,dy\\
&-\frac{\rho_H^2}{2}\int_{\Omega}\int_{\Omega} \tilde{K}_{HH}(x-y)Q(x)\cdot Q(y)\,dx\,dy-\rho_D\rho_H\int_{\Omega}\int_{\Omega} \left(\tilde{K}_{HD}(x-y)Q(x)\cdot \xi(y)+\tilde{K}_{cH}(x-y)Q(x)\cdot \xi(y)\right)\,dx\,dy.
\end{split}
\end{equation}

\subsection{Scaling, periodic domains and non-dimensionalisation}
Consider the energy \eqref{eqEnergyMacroscopic} in the case when $\Omega=\mathbb{R}^3$, with a configuration $\frac{2\pi}{\epsilon}$-periodic in each of the coordinate dimensions. Let $\mathbb{T}^3$ denote the (flat) torus in 3D with unit sides $2\pi$. The energy per unit cell is then given by 
\begin{equation}\begin{split}
&\int_{\frac{1}{\epsilon}\mathbb{T}^3}k_BT\rho_H\psi_s(Q(x))+k_BT\rho_D\psi_s(\xi(x))\,dx-\frac{\rho_D^2}{2}\int_{\frac{1}{\epsilon}\mathbb{T}^3}\int_{\mathbb{R}^3} \tilde{K}_{DD}(x-y)\xi(x)\cdot \xi(y)+\tilde{K}_{cD}(x-y)\xi(x)\cdot \xi(y)\,dx\,dy\\
&-\frac{\rho_H^2}{2}\int_{\frac{1}{\epsilon}\mathbb{T}^3}\int_{\mathbb{R}^3} \tilde{K}_{HH}(x-y)Q(x)\cdot Q(y)\,dx\,dy-\rho_D\rho_H\int_{\frac{1}{\epsilon}\mathbb{T}^3}\int_{\mathbb{R}^3} \left(\tilde{K}_{HD}(x-y)Q(x)\cdot \xi(y)+\tilde{K}_{cH}(x-y)Q(x)\cdot \xi(y)\right)\,dx\,dy.
\end{split}
\end{equation}
Note in the double integrals, the inner integral is over all of $\mathbb{R}^3$ as molecules interact not only with molecules in their unit cell, but all cells due to the non-locality. We now perform a change of variables, $x=\frac{1}{\epsilon}x'$, $y=\frac{1}{\epsilon}y'$, $Q'(x')=Q(x)$, $\xi'(x')=\xi(x)$. Then $Q',\xi'$ are $2\pi$-periodic in the coordinate directions. The energy can thus be written as 
\begin{equation}\begin{split}
\tilde{\mf}=&\frac{1}{\epsilon^3}\int_{\mathbb{T}^3}k_BT\rho_H\psi_s(Q'(x'))+k_BT\rho_D\psi_s(\xi'(x'))\,dx'\\
&-\frac{\rho_D^2}{2\epsilon^6}\int_{\mathbb{T}^3}\int_{\mathbb{R}^3} \tilde{K}_{DD}\left(\frac{x'-y'}{\epsilon}\right)\xi'(x')\cdot \xi'(y')+\tilde{K}_{cD}\left(\frac{x'-y'}{\epsilon}\right)\xi'(x')\cdot \xi'(y')\,dx'\,dy'\\
&-\frac{\rho_H^2}{2\epsilon^6}\int_{\mathbb{T}^3}\int_{\mathbb{R}^3} \tilde{K}_{HH}\left(\frac{x'-y'}{\epsilon}\right)Q'(x')\cdot Q'(y')\,dx'\,dy'\\
&-\frac{\rho_D\rho_H}{\epsilon^6}\int_{\mathbb{T}^3}\int_{\mathbb{R}^3} \left(\tilde{K}_{HD}\left(\frac{x'-y'}{\epsilon}\right)Q'(x')\cdot \xi'(y')+\tilde{K}_{cH}\left(\frac{x'-y'}{\epsilon}\right)Q'(x')\cdot \xi'(y)\right)\,dx'\,dy'.
\end{split}
\end{equation}
We will consider an asymptotically dilute regime, so that $\rho_D=\rho_0\rho_H\epsilon$ for some $\rho_0>0$, independent of $\epsilon$. Then we consider the rescaled free energy $\frac{\tilde{\epsilon\mf}}{\rho_H k_BT}$, which satisfies 
\begin{equation}
\begin{split}
\frac{\epsilon\tilde{\mf}}{\rho_H k_BT}=&\frac{1}{\epsilon^2}\int_{\mathbb{T}^3}\psi_s(Q'(x'))+\rho_0\epsilon\psi_s(\xi'(x'))\,dx'\\
&-\frac{\rho_0^2\rho_H}{2\epsilon^3k_BT}\int_{\mathbb{T}^3}\int_{\mathbb{R}^3} \tilde{K}_{DD}\left(\frac{x'-y'}{\epsilon}\right)\xi'(x')\cdot \xi'(y')+\tilde{K}_{cD}\left(\frac{x'-y'}{\epsilon}\right)\xi'(x')\cdot \xi'(y')\,dx'\,dy'\\
&-\frac{\rho_H}{2\epsilon^5k_BT}\int_{\mathbb{T}^3}\int_{\mathbb{R}^3} \tilde{K}_{HH}\left(\frac{x'-y'}{\epsilon}\right)Q'(x')\cdot Q'(y')\,dx'\,dy'\\
&-\frac{\rho_0\rho_H}{\epsilon^4k_BT}\int_{\mathbb{T}^3}\int_{\mathbb{R}^3} \left(\tilde{K}_{HD}\left(\frac{x'-y'}{\epsilon}\right)Q'(x')\cdot \xi'(y')+\tilde{K}_{cH}\left(\frac{x'-y'}{\epsilon}\right)Q'(x')\cdot \xi'(y)\right)\,dx'\,dy'.
\end{split}
\end{equation}

We now redefine our operators into a dimensionless analogue, as 
\begin{equation}
\begin{array}{c  c c}
\frac{\rho_H}{k_BT}\tilde{K}_{DD}\mapsto K_{DD}, & \frac{\rho_H}{k_BT}\tilde{K}_D^c\mapsto K_D^c, & \frac{\rho_H}{k_BT}\tilde{K}_{HH}\mapsto K_{HH},\\
\frac{\rho_H}{k_BT}\tilde{K}_{HD}\mapsto K_{HD}, & \frac{\rho_H}{k_BT}\tilde{K}_{cH}\mapsto K_{cH}
\end{array}
\end{equation}

This leads us to the energy which we will aim to minimise, 
\begin{equation}\begin{split}
\mf_\epsilon(Q,\xi)= &\int_{\mt^3} \frac{1}{\epsilon^2}\psi_s(Q)+\frac{1}{\epsilon} \rho_0\psi_s(\xi)-c_\epsilon\,dx\\
&-\frac{1}{2\epsilon^5}\int_{\mt^3}\int_{\mathbb{R}^3}K_{HH}\left(\frac{x-y}{\epsilon}\right)Q(x)\cdot Q(y)+2\rho_0\epsilon K_{DH}\left(\frac{x-y}{\epsilon}\right)Q(x)\cdot\xi(y)+\rho_0^2\epsilon^2K_{DD}\left(\frac{x-y}{\epsilon}\right)\xi(x)\cdot\xi(y)\,dy\,dx\\
&+ \frac{1}{\epsilon^5}\int_{\mt^3}\int_{\mathbb{R}^3}\rho_0\epsilon K_{cH}\left(\frac{x-y}{\epsilon}\right)\xi(x)\cdot Q(y)+\frac{1}{\epsilon^5}\rho_0^2\epsilon^2K_{cD}\left(\frac{x-y}{\epsilon}\right)\xi(x)\cdot \xi(y)\,dy\,dx
\end{split}
\end{equation}

The constant $c_\epsilon$ is to be determined later and does not affect the minima, and will be chosen so that the ``bulk" contribution to the energy has minimum zero.

We now include some technical assumptions on the interaction terms, which will be required for the precise analysis in \Cref{sectionGammaConvergence}. Note consants $C,C_1,C_2$ will be used to represent generic positive constants and may change from line to line, here and throughout the paper

\begin{assumption}\label{assumptionInteraction} The functions $K_{XY}$, $K_X^c$ are $L^1(\mathbb{R}^3)$, frame indifferent functions for all choices of $X,Y$ as $H,D$. Furthermore, they satisfy the following inequalities. There exists positive constants $C_1,C_2$, and a non-negative frame indifferent function $g\in L^1(\mathbb{R}^3)$ with $\int_{\mathbb{R}^3}|x|^2g(x)\,dx<+\infty$ so that
\begin{equation}
\begin{split}
&C_1 g(z)\left(|A|^2+|B|^2-1\right)\\
<&K_{HH}(z)A\cdot A + 2K_{HD}(z)A\cdot B + K_{DD}(z)B\cdot B+ \sqrt{\frac{2}{3}}|K_{cH}(z)||A|+ \sqrt{\frac{2}{3}}|K_{cD}||B|\\
\leq & C_2g(z)\left(|A|^2+|B|^2+1\right).\end{split}
\end{equation}
Furthermore, for $X,Y=H,D$, there exists positive constants $C_2>C_1>0$ so that
\begin{equation}\label{eqEstimate1}
C_1g(z)\text{Id}\leq K_{XY}(z) \leq C_2 g(z)\text{Id},
\end{equation}
where $\text{Id}$ denotes the identity matrix on $\text{Sym}_0(3)$ and the inequality holds as bilinear forms. Finally there exists positive constants $C>0$ so that 
\begin{equation}\label{eqEstimate2}
|K_{cX}(z)|\leq Cg(z)
\end{equation}
for $X=H,D$.

\end{assumption}
Loosely speaking, these estimates say that the symmetric parts of the interaction energies are all comparable and give rise to a positive definite bilinear form. The fact that the same function can estimate above and below is indicative of the fact that at a fixed separation $z$, the interaction energy cannot vary too wildly as the relative orientations of particles vary. We note that the estimates \eqref{eqEstimate1}, \eqref{eqEstimate2} are consistent with London dispersion forces for the achiral contributions \cite{london1930theorie} and the van der Meer expression for chiral contributions \cite{van1976molecular}, if the interaction is ``cut-off" at close distance, a technique employed in the derivation of the Maier-Saupe free energy and evocative of the representation of such interactions as being long-range in character \cite{maier1959einfache}.

Next we introduce some notation as in \cite{taylor2018oseen}.
\begin{definition}
Let $u\in L^1(\mathbb{R}^3)$, $\epsilon>0$. Then define $u^\epsilon\in L^1(\mathbb{T}^3)$ by 
\begin{equation}
u^\epsilon(x)=\frac{1}{\epsilon^3}\sum\limits_{k\in\mathbb{Z}^3} u\left(\frac{x+2\pi k }{\epsilon}\right).
\end{equation}
\end{definition}

This definition can be interpreted in a duality-like fashion, as it is equivalent to the statement that 
\begin{equation}
\frac{1}{\epsilon^3}\int_{\mathbb{R}^3}v(x)u\left(\frac{x}{\epsilon}\right)\,dx=\int_{\mathbb{T}^3}u^\epsilon(x)v(x)\,dx
\end{equation}
for all $v\in L^\infty(\mathbb{R}^3)$ which are $2\pi$-periodic in the coordinate directions, or equivalently $v\in L^\infty(\mathbb{T}^3)$.

 Thus we can write the energy as 
\begin{equation}\begin{split}\label{eqEnergy}
\mf_\epsilon(Q,\xi)= &\int_{\mt^3} \frac{1}{\epsilon^2}\psi_s(Q)+\frac{1}{\epsilon} \rho_0\psi_s(\xi)\,dx\\
&-\frac{1}{2\epsilon^2}\int_{\mt^3}\int_{\mt^3}K_{HH}^\epsilon(x-y)Q(x)\cdot Q(y)+2\rho_0\epsilon K_{DH}^\epsilon\left(x-y\right)Q(x)\cdot\xi(y)+\rho_0^2\epsilon^2K_{DD}^\epsilon\left(x-y\right)\xi(x)\cdot\xi(y)\,dy\,dx\\
&+ \frac{1}{\epsilon^2}\int_{\mt^3}\int_{\mt^3}\rho_0\epsilon K_{cH}^\epsilon\left(x-y\right)\xi(x)\cdot Q(y)+\rho_0^2\epsilon^2K_{cD}^\epsilon\left(x-y\right)\xi(x)\cdot \xi(y)\,dy\,dx
\end{split}
\end{equation} 
The natural function spaces for $Q,\xi$ are $L^\infty(\mathbb{T}^3,\mqc)$, as the constraints $\text{Tr}(Q)=0$ and $\lambda_{\min}(Q)>-\frac{1}{3}$ imply ${||Q||_\infty<\frac{2}{3}}$ and ${||\xi||_\infty<\frac{2}{3}}$

\subsection{Re-writing the energy}\label{subsecRewrite}

We now turn to writing the given energy in a form more convenient for our analysis.

For the achiral terms involving the operators $K_{XY}$, we note the general identity for any symmetric bilinear form and vectors $A(x),A(y),B(x),B(y)$ 
\begin{equation}
\langle A(x)-A(y),B(x)-B(y)\rangle = \langle A(x),B(x)\rangle + \langle A(y),B(y)\rangle -\langle A(x),B(y)\rangle -\langle A(y),B(x)\rangle
\end{equation}
This implies that the achiral integral terms can be written as 
\begin{equation}
\begin{split}
&\int_{\mt^3}\int_{\mt^3}K_{XY}^\epsilon(x-y)A(x)\cdot B(y)\,dx\,dy \\
=& \int_{\mt^3}\int_{\mt^3} K_{XY}^\epsilon(x-y)A(x)\cdot B(x)+\frac{1}{2}K_X^\epsilon(x-y)\big(A(x)-A(y)\big)\cdot \big(B(x)-B(y)\big)\,dx\,dy\\
=&\int_{\mt^3} k_{XY}^0 A\cdot B\,dx+\frac{1}{2}\int_{\mt^3}\int_{\mt^3}K_{XY}^\epsilon(x-y)\big(A(x)-A(y)\big)\cdot \big(B(x)-B(y)\big)\,dy\,dx.
\end{split}
\end{equation}
Here the tensor $k_{XY}^0$ is defined by 
\begin{equation}
k_{XY}^0=\int_{\mathbb{T}^3}K_{XY}^\epsilon(z)\,dz=\int_{\mathbb{R}^3}K_{XY}(z)\,dz.
\end{equation}
By symmetry, as $k_{XY}^0$ defines a linear, frame invariant function of a trace-free matrix, it must hold that $k_{XY}^0$ is a multiple of the identity operator \cite{smith1971isotropic}. In particular, we will abuse notation and frequently write $ k_{XY}^0A\cdot B = k_{XY}^0 (A\cdot B)$, viewing $k_{XY}^0$ as a scalar. 
Turning to the chiral terms $K_{cX}$, we may write
\begin{equation}\begin{split}
&\int_{\mt^3}\int_{\mt^3}K_{cX}^\epsilon(x-y)A(x)\cdot B(y)\,dy\,dx\\
=& \int_{\mt^3}\int_{\mt^3}K_{cX}^\epsilon(x-y)(A(x)-A(y))\cdot B(y)+K_{cX}^\epsilon(x-y)A(y)\cdot B(y)\,dy\,dx\\
=& \int_{\mt^3}\int_{\mt^3}K_{cX}^\epsilon(x-y)(A(x)-A(y))\cdot B(y)\,dy\,dx+\int_{\mt^3}\left(\int_{\mt^3}K_{cX}^\epsilon(x-y)\,dx\right) A(y)\cdot B(y)\,dy\\
=&\intt \intt K_{cX}^\epsilon(x-y)(A(x)-A(y))\cdot B(y)\,dy\,dx.
\end{split}
\end{equation}
The last equality follows from the antisymmetry of $K_{cX}$, so that 
\begin{equation}
\intt K_{cX}^\epsilon(z)\,dz=-\intt K_{cX}^\epsilon(-z)\,dz=0.
\end{equation}

Throughout we will use the notation that $A^{\otimes 2}=A\otimes A$ for a tensor $A$, so if $A$ lives in an inner product space $V$, and $T$ is an operator from $V$ to itself, $T\cdot A^{\otimes 2}=TA\cdot A$. Generally, for any two tensors of equal rank, $\cdot$ will denote their inner product.

We now define the functional which we will consider for periodic boundary conditions.

\begin{definition}\label{defEnergy}
Let $\mf_\epsilon: L^\infty(\mathbb{T}^3;\mqc)^2\to\mathbb{R}\cup\{+\infty\}$ be defined by  
\begin{equation}\begin{split}
\mf_\epsilon(Q,\xi)= &\int_{\mt^3} \frac{1}{\epsilon^2}\psi_s(Q)+\frac{1}{\epsilon} \rho_0\psi_s(\xi)-\frac{k_{HH}^0}{2}|Q|^2-\frac{\rho_0}{\epsilon} k_{HD}^0 Q\cdot \xi -\frac{\rho_0^2k_{DD}^0}{2}|\xi|^2-c_\epsilon\,dx\\
&+\frac{1}{4\epsilon^2}\int_{\mt^3}\int_{\mt^3}K_{HH}^\epsilon(x-y)\cdot\big(Q(x)-Q(y)\big)^{\otimes 2} +\rho_0^2\epsilon^2K_{DD}^\epsilon\left(x-y\right)\big(\xi(x)-\xi(y)\big)^{\otimes 2}\,dy\,dx\\
&+\frac{\rho_0}{2\epsilon}\int_{\mathbb{T}^3}\int_{\mt^3} K_{DH}^\epsilon\left(x-y\right)\big(Q(x)-Q(y)\big)\cdot\big(\xi(x)-\xi(y)\big)\,dy\,dx\\
&+ \frac{1}{\epsilon^2}\int_{\mt^3}\int_{\mt^3}\rho_0\epsilon K_{cH}^\epsilon\left(x-y\right)\xi(x)\cdot \big(Q(y)-Q(x)\big)+\rho_0^2\epsilon^2K_{cD}^\epsilon\left(x-y\right)\big(\xi(x)-\xi(y)\big)\cdot \xi(y)\,dy\,dx.
\end{split}
\end{equation}
The constant $c_\epsilon$ is defined by 
\begin{equation}\label{eqCEpsilon}
c_\epsilon = \min\limits_{Q,\xi}\frac{1}{\epsilon^2}\left(\psi_s(Q)-\frac{k_{HH}^0}{2}|Q|^2\right)+\frac{\rho_0}{\epsilon}\left(\psi_s(\xi)-k_{HD}^0Q\cdot \xi\right)
\end{equation}
\end{definition}

\section{Heuristic gradient expansion}\label{sectionHeuristic}

Before proceeding with a precise $\Gamma$-convergence argument, we perform a formal argument, which will be consistent with \Cref{thmPeriodicGamma}

If $A,B:\mathbb{T}^3\to\text{Sym}_0(3)$ are sufficiently regular, we expect that 
\begin{equation}\begin{split}
&\frac{1}{\epsilon^2}\intt\intt K^\epsilon_{XY}(x-y)(A(x)-A(y))\cdot (B(x)-B(y))\,dy\,dx\\
=& \intt\int_{\mathbb{R}^3}\frac{1}{\epsilon^5}K_{XY}\left(\frac{x-y}{\epsilon}\right)(A(x)-A(y))\cdot (B(x)-B(y))\,dy\,dx\\
=& \intt\int_{\mathbb{R}^3}\frac{1}{\epsilon^2}K_{XY}\left(z\right) (A(x)-A(x+\epsilon z))\cdot (B(x)-B(x+\epsilon z))\,dz\,dx\\
=& \intt\int_{\mathbb{R}^3}|z|^2K_{XY}\left(z\right)\frac{A(x)-A(x+\epsilon z)}{\epsilon |z|}\cdot \frac{B(x)-B(x+\epsilon z)}{\epsilon|z|}\,dz\,dx\\
\approx & \intt\int_{\mathbb{R}^3}|z|^2K_{XY}\left(z\right) \bigg((\hat{z}\cdot\nabla) A(x)\bigg)\cdot\bigg( (\hat{z}\cdot \nabla)B(x)\bigg)\,dz\,dx\\
=& 2\intt L_{XY}\nabla A(x)\cdot \nabla B(x)\,dx,
\end{split}
\end{equation}
for an appropriate operator $L_{XY}$, given as an integral 
\begin{equation}
L_{XY}=\frac{1}{2}\int_{\mathbb{R}^3}K_{XY}(z)\otimes z\otimes z\,dz
\end{equation}
Similarly, for the chiral terms, we expect 
\begin{equation}
\begin{split}
&\frac{1}{\epsilon}\intt\intt K_{cX}^\epsilon(x-y)A(x)\cdot( B(x)-B(y))\,dy\,dx\\
=& \intt\int_{\mathbb{R}^3} \frac{1}{\epsilon^4}K_{cX}\left(\frac{x-y}{\epsilon}\right)A(x)\cdot( B(x)-B(y))\,dy\,dx\\
=&  \frac{1}{\epsilon}\intt\int_{\mathbb{R}^3} K_{cX}\left(z\right)A(x)\cdot( B(x)-B(x+\epsilon z))\,dz\,dx\\
\approx &-\intt\int_{\mathbb{R}^3}|z| K_{cX}\left(z\right)A(x)\cdot(\hat{z}\cdot \nabla)B(x)\,dz\,dx\\
=& \intt V_{cX}A(x)\cdot \nabla B(x)\,dx,
\end{split}
\end{equation}
where the operator 
\begin{equation}
V_{cX}=-\int_{\mathbb{R}^3}K_{cX}(z)\otimes z\,dz.
\end{equation}
Then, we may substitute these gradient approximations into the energy, to give 
\begin{equation}\label{eqLandauCholesteric}
\begin{split}
\mf_\epsilon(Q,\xi)\approx & \intt\frac{1}{\epsilon^2}\left(\psi_s(Q(x))-\frac{k_{HH}^0}{2}|Q(x)|^2\right)+\frac{\rho_0}{\epsilon}\left(\psi_s(\xi(x))-k_{HD}^0Q(x)\cdot \xi(x)\right)\,dx\\
& + \intt \frac{1}{2}L_{HH}\nabla Q(x)\cdot \nabla Q(x) +\rho_0V_{cH}\xi(x)\cdot \nabla Q(x)-\frac{k_{DD}^0\rho_0^2}{2}|\xi(x)|^2\,dx\\
&+ \epsilon \intt \rho_0L_{HD}\nabla Q(x)\cdot \nabla \xi(x)+\rho^2_0V_{cD}\xi(x)\cdot \nabla \xi(x)\,dx\\
&+\epsilon^2\intt\frac{\rho^2_0}{2}L_{DD}\nabla \xi(x)\cdot \nabla \xi(x)\,dx. 
\end{split}
\end{equation}
We discard all terms of order $\epsilon$ and $\epsilon^2$. The leading order term, of order $\frac{1}{\epsilon^2}$ is the bulk energy of $Q$, $\psi_s(Q)-\frac{k_{HH}^0}{2}|Q|^2$. As this is leading order, we assume our solutions are well approximated by restricting $Q$ to be in the minimising set of this bulk energy, so we take $Q=s_0\left(n\otimes n-\frac{1}{3}I\right)$ for an appropriate constant $s_0$ and $n:\mathbb{T}^3\to\mathbb{S}^2$, as per \Cref{propSingPotential}. 

At the next order, we have a bulk-type energy for the dopant, $\psi_s(\xi)-k_{HD}^0\xi\cdot Q$. We assume this too must be well approximated by restricting $\xi$ to the minimising set of this energy. As $Q$ is prescribed, by symmetry and convexity arguments we can show that $\xi(x)=\frac{s_c}{s_0}Q(x)$ for some constant $s_c$. There is a precise explanation of the symmetry arguments available in \Cref{theoremCompactness}.

Now, we substitute these assumptions into the energy, and obtain that 
\begin{equation}
\mf_\epsilon \left(Q,\frac{s_c}{s_0}Q\right)\approx \intt \frac{1}{2}L_{HH}\nabla Q(x)\cdot \nabla Q(x) +\frac{s_c}{s_0}\rho_0V_{cH}Q(x)\cdot \nabla Q(x)\,dx +c_0.
\end{equation}

By further symmetry arguments (see \Cref{sectionCoefficients}), if $n\in W^{1,2}(\mathbb{T}^3,\mathbb{S}^2)$, we can re-write this energy as 
\begin{equation}\label{eqOFApprox}
\mf_\epsilon \left(Q,\frac{s_c}{s_0}Q\right)\approx \frac{s_0^2}{2}\intt K_{11}(\nabla \cdot n(x))^2+K_{22}(n(x)\cdot \nabla\times n(x)+q)^2+K_{33}|n(x)\times \nabla \times n(x)|^2\,dx,
\end{equation}
where $K_{11},K_{22},K_{33}$ are constants relating to the energy penalisation for splay, twist and bend (respectively) of $n$. $q$ is a pseudo-scalar which dictates the tendency to twist. If $q=0$, the ground state is given by configurations constant in space. When $q\neq 0$, the ground state is a helical structure with pitch $\frac{\pi}{q}$. 

This formal argument bears similarities to that posed by Osipov and Kuball \cite{osipov2001helical}, who provide an expression for HTP through similar dispersion energies and a gradient expansion. The most significant development of this work from theirs is that in this work we prove various statements taken as ansatzes or assumptions in theirs, such as the host Q-tensor being uniaxial with fixed order parameter, negligibility of Dopant-Dopant interactions, Host-Dopant alignment and the applicability of a (weak)-gradient theory. In order to bypass making these assumptions we take a simpler constitutive equation on the interaction energies and a scaling limit, under which they may be proven. As the models are different the results cannot be fully consistent, however the results we obtain here are qualitatively consistent with theirs.

\begin{remark}
This formal argument provides insight to some of the more qualitative features of the model. Firstly, we see that all dopant-dopant interactions do not contribute to the limiting energy. Similarly, we see no change to the Frank constants themselves, nor the equilibrium order parameter of the Q-tensor compared to if we have a pure host system. These results are of course artifacts of the dilute regime we consider, and this formal analysis suggests there are order $\epsilon$ correction to the Frank constants depending on the Host-Dopant interactions, an order $\epsilon$ correction to the intrinsic twist $q$ depending on the Dopant-Dopant chiral interaction, and order $\epsilon^2$ corrections to the Frank constants from the Dopant-Dopant interactions. Further corrections should be expected for $\epsilon>0$ as the order parameters may move away from the bulk minimisers to relax elastic contributions to the energy. We will not focus on these corrections within this work, and proceed to obtain the relationship \eqref{eqOFApprox} in a rigorous sense. 
\end{remark}

\section{$\Gamma$-convergence}\label{sectionGammaConvergence}

\subsection{Preliminary results}

The estimates in \Cref{assumptionInteraction} translate into estimates on the non-local part of the energy, in the form 
\begin{equation}\begin{split}\label{eqCoercivityEstimate}
& C_1\intt\intt g^\epsilon(x-y)\left(\frac{|Q(x)-Q(y)|^2}{\epsilon^2}+\rho_0^2|\xi(x)-\xi(y)|^2-1\right)\,dy\,dx\\
\leq &-\frac{1}{4\epsilon^2}\int_{\mt^3}\int_{\mt^3}K_{HH}^\epsilon(x-y)\cdot\big(Q(x)-Q(y)\big)^{\otimes 2} +\rho_0^2\epsilon^2K_{DD}^\epsilon\left(x-y\right)\big(\xi(x)-\xi(y)\big)^{\otimes 2}\,dy\,dx\\
&+\frac{\rho_0}{2\epsilon}\int_{\mathbb{T}^3}\int_{\mt^3} K_{DH}^\epsilon\left(x-y\right)\big(Q(x)-Q(y)\big)\cdot\big(\xi(x)-\xi(y)\big)\,dy\,dx\\
&+ \frac{1}{\epsilon^2}\int_{\mt^3}\int_{\mt^3}\rho_0\epsilon K_{cH}^\epsilon\left(x-y\right)\xi(x)\cdot \big(Q(y)-Q(x)\big)+\rho_0^2\epsilon^2K_{cD}^\epsilon\left(x-y\right)\big(\xi(x)-\xi(y)\big)\cdot \xi(y)\,dy\,dx\\
\leq & C_2\intt\intt g^\epsilon(x-y)\left(\frac{|Q(x)-Q(y)|^2}{\epsilon^2}+\rho_0^2|\xi(x)-\xi(y)|^2+1\right)\,dy\,dx,
\end{split}
\end{equation}

We recall some compactness and continuity results from \cite{taylor2018oseen}

\begin{proposition}\label{propCompactBilinear}
Let $A_\epsilon\in L^\infty(\mt^3,\overline{\mathcal{Q}})$ be so that there exists some $M\in\mathbb{R}$ with 
\begin{equation}
\frac{1}{\epsilon^2}\intt\intt g^\epsilon(z)|A_\epsilon(x)-A_\epsilon(y)|^2\,dx\,dy<M
\end{equation}
for all $\epsilon>0$. Then there exists some sequence $\epsilon_j\to 0$ and $A\in W^{1,2}(\mt^3, \overline{\mathcal{Q}})$ so that $A_{\epsilon_j}=A_j\overset{L^2}{\to}A$. 
\end{proposition}

\begin{proposition}\label{propLiminfNonlocal}

Let $K$ satisfy the estimates in \Cref{assumptionInteraction}. If $\epsilon_j\to 0$, $A_{\epsilon_j}\in L^\infty(\mt^3,\overline{\mathcal{Q}})$ converge in $L^2$ to some $A\in W^{1,2}(\mt^3,\overline{\mathcal{Q}})$, then 
\begin{equation}
\liminf\limits_{j\to\infty} \intt\intt \frac{1}{\epsilon^2}K^\epsilon(x-y)\left(A(x)-A(y)\right)^{\otimes 2}\,dy\,dx \geq \intt L\nabla A(x)\cdot \nabla A(x)\,dx,
\end{equation}
where 
\begin{equation}
L\nabla A\cdot \nabla A = \int_{\mathbb{R}^3} K(z)\frac{\partial A}{\partial x_\beta}\cdot \frac{\partial A}{\partial x_\alpha}z_\alpha z_\beta\,dz.
\end{equation}\end{proposition}

\begin{proposition}\label{propLimsupNonlocal}

Let $K$ satisfy the estimates in \Cref{assumptionInteraction}. If $\epsilon_j\to 0$, $A_{\epsilon_j}\in W^{1,2}(\mt^3,\overline{\mathcal{Q}})$ converge strongly in $W^{1,2}$ to some $A$, then 
\begin{equation}
\lim\limits_{j\to\infty} \intt\intt \frac{1}{\epsilon^2}K^\epsilon(x-y)\left(A(x)-A(y)\right)^{\otimes 2}\,dy\,dx= \intt L\nabla A(x)\cdot \nabla A(x)\,dx,
\end{equation}
where 
\begin{equation}
L\nabla A\cdot \nabla A = \int_{\mathbb{R}^3} K(z)\frac{\partial A}{\partial x_\beta}\cdot \frac{\partial A}{\partial x_\alpha}z_\alpha z_\beta\,dz.
\end{equation}\end{proposition}

Furthermore, we recall the definition of $\Gamma$-convergence \cite{braides2002gamma}, in a simplified form relevant for this work.

\begin{definition}
Let $F_\epsilon:V\to\mathbb{R}\cup\{+\infty\}$, with $V$ a normed vector space. Then we say that $F_\epsilon$ $\Gamma$-converges to $F:V\to\mathbb{R}\cup\{+\infty\}$ if the following hold:
\begin{itemize}
\item (Liminf inequality): For every sequence $v_\epsilon \to v$, $\liminf\limits_{\e\to 0}F_\e(v_\e)\geq F(v)$.
\item (Limsup inequality): For every $v\in V$, there exists a sequence $v_\e\to v$ with $\lim\limits_{\e\to 0 }F_\e(v_\e)=F(v)$. 
\end{itemize}
\end{definition}

Furthermore, if we have $\Gamma$-convergence and a coercivity property, then we may obtain the fundamental theorem of $\Gamma$-convergence, which motivates its definition.  
\begin{proposition}
Assume that $F_\e$ $\Gamma$-converges to $F$. Furthermore, assume that $F_\e$ is equicoercive, in the sense that if $F_\e(v_\e)$ is uniformly bounded for a sequence $(v_\e)_{\e>0}$, there exists a subsequence $v_j=v_{\e_j}$ and $v\in V$ so that $v_j\to v$. Then 
\begin{enumerate}
\item $\lim\limits_{\e\to 0}\inf\limits_{V} F_\e=\min\limits_{V} F$.
\item If $v_\e$ is a sequence of approximate minimisers, so that $F_\e(v_\e)-\inf\limits_{V}F_\e\to 0$, then there exists a subsequence $v_j=v_{\e_j}$ and $v$ with $F(v)=\min\limits_{V}F$ so that $v_j \to v$. 
\end{enumerate}
In particular, minimisers of $F$ exist, and minimisers of $F_\e$ converge, up to subsequences, to minimisers of $F$.
\end{proposition}
In essence, this result justifies the interpretation that $\Gamma$-convergence as an appropriate method of convergence for minimisation problems. In this section, we proceed to show the equicoercivity property (\Cref{subsecCompact}), the liminf inequality (\Cref{subsecLiminf}) and limsup inequality (\Cref{subsecLimsup}).

\subsection{Compactness}\label{subsecCompact}

\begin{definition}
Let $\mm$ define the minimising set of the bulk energy for the host Q-tensor in the absence of dopant. Explicitly, 
\begin{equation}
\mm= \left\{ Q\in \mq : \psi_s(Q)-\frac{k_{HH}^0}{2}|Q|^2=\min\limits_{\tilde Q}\psi_s(\tilde Q)-\frac{k_{HH}^0}{2}|\tilde Q|^2\right\}.
\end{equation}
 Furthermore, we define \begin{equation}c_0=\min\limits_{\tilde Q}\psi_s(\tilde Q)-\frac{k_{HH}^0}{2}|\tilde Q|^2\end{equation}
\end{definition}

\begin{assumption}
We take $k_{HH}^0$ sufficiently large so that 
\begin{equation}
\mm=\left\{s_0\sigma(n):n\in\mathbb{S}^2\right\},
\end{equation}
where $s_0>0$ is dependent on $k_{HH}^0$.
\end{assumption}

\begin{proposition}\label{propQL2}
Let $Q_\e,\xi_\e \in L^\infty(\mt^3,\mqc)$. If $\mf_\epsilon(Q_\epsilon,\xi_\epsilon)$ is uniformly bounded, then there exists some subsequence $\epsilon_j\to 0$ and $Q\in W^{1,2}(\mt^3,\mqc)$ so that $Q_{\epsilon_j}\overset{L^2}{\to} Q$.
\end{proposition}
\begin{proof}
The coercivity estimate \eqref{eqCoercivityEstimate} implies that if $\mathcal{F}_\epsilon(Q_\epsilon,\xi_\epsilon)$ is bounded, then 
\begin{equation}
\frac{1}{\epsilon^2}\intt\intt g^\epsilon(x-y)|Q_\epsilon(x)-Q_\e(y)|^2
\end{equation}
is bounded, so by \Cref{propCompactBilinear}, we have compactness. 
\end{proof}

\begin{proposition}
Let $Q_\e,\xi_\e \in L^\infty(\mt^3,\mqc)$. If $\mf_\epsilon(Q_\epsilon,\xi_\epsilon)$ is uniformly bounded, then any $L^2$ cluster point of $Q_\epsilon$ must be pointwise almost everywhere $\mm$ valued. Furthermore, 
\begin{equation}
\frac{c_\epsilon^2-c_0}{\epsilon}\leq -\rho_0\ln\int_{\mathbb{S}^2}\exp\left(k_{DH}^0Q_0p\cdot p\right)\,dp=\rho_0\min\limits_{\xi}\psi_s(\xi)-k_{DH}^0Q_0\cdot\xi,
\end{equation}
where $Q_0\in\mm$ and $c_\epsilon$ is as in \eqref{eqCEpsilon}.
\end{proposition}
\begin{proof}
To show that $Q$ is $\mm$ valued almost everywhere it suffices to show that $\intt \psi_s(Q_\epsilon)-\frac{k_{HH}^0}{2}|Q_\epsilon|^2-c_0\,dx\to 0$, as we may take a pointwise a.e. converging subsequence and apply Fatou's lemma. To see this, first we estimate $c_\epsilon$. As the map $\xi\mapsto\psi_s(\xi)-k_{DH}^0\xi\cdot Q$ is strictly convex for fixed $Q$, we can readily describe the unique minimiser by the critical point condition, $\xi=\Lambda^{-1}(k_{DH}^0Q)$. This gives the minimum value as 
\begin{equation}
\min\limits_{\xi}\left(\psi_s(\xi)-k_{DH}^0\xi\cdot Q\right)=-\ln\int_{\mathbb{S}^2}\exp(k_{DH}^0Qp\cdot p)\,dp.
\end{equation} Then if $Q_0$ is a minimiser of $\psi_s(Q)-\frac{k^0_{HH}}{2}|Q|^2$,
\begin{equation}
\begin{split}
c_\epsilon\epsilon^2 =&   \min\limits_{Q,\xi} \psi_s(Q)-\frac{k_{HH}^0}{2}|Q|^2+\rho_0\epsilon(\psi_s(\xi)-k_{DH}^0\xi\cdot Q)\\
=& \min\limits_{Q} \psi_s(Q)-\frac{k_{HH}^0}{2}|Q|^2-\rho_0\epsilon\ln\intt\exp(k_{DH}^0Qp\cdot p)\,dp\\
\leq & \left(\psi_s(Q_0)-\frac{k_{HH}^0}{2}|Q_0|^2\right)-\rho_0\epsilon\ln\intt\exp(k_{DH}^0Q_0p\cdot p)\,dp\\
=& c_0 + C\epsilon.
\end{split}
\end{equation}
Now we turn to the energy estimate, which immediately implies
\begin{equation}
\mf_\epsilon(Q_\epsilon,\xi_\epsilon)\geq \intt \frac{1}{\epsilon^2}\left(\psi_s(Q_\epsilon)-\frac{k_{HH}^0}{2}|Q_\epsilon|^2\right)+\frac{\rho_0}{\epsilon}\left(\psi_s(\xi_\epsilon)-k_{DH}^0Q_\epsilon\cdot\xi_\epsilon\right)-c_\epsilon -C
\end{equation}
This then implies that for some $M>0$
\begin{equation}
\begin{split}
M\epsilon^2\geq &\intt \psi_s(Q_\epsilon)-\frac{k_{HH}^0}{2}|Q_\epsilon|^2+\rho_0\epsilon\left(\psi_s(\xi_\epsilon)-k_{DH}^0Q_\epsilon\cdot\xi_\epsilon\right)-\epsilon^2c_\epsilon \,dx\\
\geq & \intt \psi_s(Q_\epsilon)-\frac{k_{HH}^0}{2}|Q_\epsilon|^2+\rho_0\epsilon\min\limits_{\xi}\left(\psi_s(\xi)-\sqrt{\frac{2}{3}}k_{DH}^0|\xi|\right)-c_0-C\epsilon \,dx\\
\Rightarrow O(\epsilon)=& \intt \psi_s(Q_\epsilon)-\frac{k_{HH}^0}{2}|Q_\epsilon|^2-c_0 \,dx.\\
\end{split}
\end{equation}
Employing Fatou's lemma, this gives for any subsequence with $Q_{\epsilon_j}\overset{L^2}{\to} Q$ for some $Q$, that we may take a further subsequence $Q_k=Q_{\epsilon_{j_k}}$ converging pointwise, and 
\begin{equation}
\begin{split}
0\geq &\liminf\limits_{k\to\infty}\intt \psi_s(Q_k)-\frac{k_{HH}^0}{2}|Q_k|^2-c_0 \,dx\\
\geq & \intt \psi_s(Q)-\frac{k_{HH}^0}{2}|Q|^2-c_0 \,dx\geq 0,
\end{split}
\end{equation}
thus as the integrand $\psi_s(Q)-\frac{k^0_{HH}}{2}|Q|^2-c_0$ is non-negative and integrates to zero, it must equal zero almost everywhere, hence $Q\in \mm$ almost everywhere.

\end{proof}

\begin{theorem}\label{theoremCompactness}
Let $Q_\epsilon,\xi_\epsilon\in L^\infty(\mt^3,\mqc)$. If $\mf_\epsilon (Q_\epsilon,\xi_\epsilon)$ is uniformly bounded, there exists some $Q_0\in W^{1,2}(\mt^3,\mm)$ and a subsequence $Q_j=Q_{\epsilon_j},\xi_j=\xi_{\epsilon_j}$ so that $Q_j \to Q_0$ in $L^2$, and $\xi_j \to \frac{s_c}{s_0}Q_0$ in $L^2$. The constant $s_c$ satisfies \begin{equation}
s_c=\frac{3}{2}\sigma(n)\cdot\Lambda^{-1}(k_{HD}^0s_0\sigma(n)),
\end{equation}
where $n\in\mathbb{S}^2$ is arbitrary.
\end{theorem}
\begin{proof}
First we will show that there exists a subsequence $\xi_{\epsilon_j}\overset{*}{\rightharpoonup} \xi=\frac{s_c}{s_0}Q$ in $L^\infty$, then this will be used to prove estimates to show strong $L^2$ convergence.

Similarly to before,  we have the energy estimate which tells us 
\begin{equation}
\begin{split}
M\epsilon\geq &\intt \frac{1}{\epsilon}\left(\psi_s(Q_\epsilon)-\frac{k_{HH}^0}{2}|Q_\epsilon|^2\right)+\rho_0\left(\psi_s(\xi_\epsilon)-k_{DH}^0Q_\epsilon\cdot\xi_\epsilon\right)-\epsilon c_\epsilon \,dx\\
\geq & \intt\rho_0\left(\psi_s(\xi_\epsilon)-k_{DH}^0Q_\epsilon\cdot\xi_\epsilon\right)+\frac{c_0}{\epsilon}-\epsilon c_\epsilon \,dx\\
=& \intt\rho_0\left(\psi_s(\xi_\epsilon)-k_{DH}^0Q_0\cdot\xi_\epsilon\right)+k_{DH}^0\rho_0\xi_\e\cdot(Q_0-Q_\epsilon)+\frac{c_0-\epsilon^2c_\epsilon}{\epsilon} \,dx\\
\geq & \intt\rho_0\left(\psi_s(\xi)-k_{DH}^0Q_0\cdot\xi_\epsilon\right)-\rho_0\min\limits_{\tilde\xi}\left(\psi_s(\tilde\xi_\epsilon)-k_{DH}^0Q_0\cdot\tilde\xi_\epsilon\right)+k_{DH}^0\rho_0\xi_\epsilon\cdot(Q_0-Q_\epsilon)\,dx,\\
\end{split}
\end{equation}
where $Q_0\in L^\infty(\mt^3,\overline{\mathcal{Q}})$ is arbitrary. 
Now we take a subsequence so that $\epsilon_j\to 0$, $\xi_{\epsilon_j}=\xi_j \overset{*}{\rightharpoonup}\xi$ in $L^\infty$, and $Q_{\epsilon_j}=Q_j\to Q$ in $L^2$, which is permitted as $||\xi||_\infty<\frac{2}{3}$ and \Cref{propQL2}.  Furthermore, take $Q_0=Q$. Then 
\begin{equation}\begin{split}
0\leq&\intt\rho_0\left(\psi_s(\xi_j)-Q\cdot\xi_j\right)-\min\limits_{\tilde\xi}\left(\psi_s(\tilde\xi)-k_{DH}^0Q\cdot\tilde\xi\right)\,dx\\
\leq & M\epsilon-\intt \rho_0 k_{DH}^0\xi_\epsilon\cdot(Q-Q_\epsilon)\\
\leq &M\epsilon + \rho_0k_{DH}^0||\xi_\epsilon||_2||Q-Q_\epsilon||_2\to 0
\end{split}
\end{equation}

Now, as $\psi_s$ is strictly convex, this yields 
\begin{equation}\begin{split}
0=&\liminf\limits_{j\to\infty} \intt\rho_0\left(\psi_s(\xi_j)-Q\cdot\xi_j\right)-\rho_0\min\limits_{\tilde\xi}\left(\psi_s(\tilde\xi)-k_{DH}^0Q\cdot\tilde\xi\right)\,dx\\
&\geq \intt\rho_0\left(\psi_s(\xi)-Q\cdot\xi\right)-\rho_0\min\limits_{\tilde\xi}\left(\psi_s(\tilde\xi)-k_{DH}^0Q\cdot\tilde\xi\right)\,dx
\end{split}
\end{equation}
This means that $\xi$ must be pointwise in the minimum energy wells of $\tilde{\xi}\mapsto \psi_s(\tilde{\xi})-Q\cdot \tilde{\xi}$. This is a strictly convex function so its minimising set must be a single point, and by taking the derivative we see that almost everywhere
\begin{equation}\label{eqXiCritical}
\xi=\Lambda^{-1}(k_{DH}^0Q(x)).
\end{equation} More so, as $Q$ is $\mm$ valued almost everywhere, this means that we can write $Q=s_0\left(n\otimes n-\frac{1}{3}I\right)$. Since $\Lambda^{-1}$ is frame invariant, this implies $\xi$ is also uniaxial with director $n$. This implies $\xi=s_c\sigma(n)$ for some $s_c$, which can readily be found by taking the inner product of \eqref{eqXiCritical} against $\frac{3}{2}\sigma(n)$, giving 
\begin{equation}
s_c=\frac{3}{2}\sigma(n)\cdot (s_c\sigma(n))=\frac{3}{2}\sigma(n)\cdot \Lambda^{-1}(k_{DH}^0s_0\sigma(n)).
\end{equation} Alternatively, we can write $\xi = \frac{s_c}{s_0}Q$.

To show that $\xi$ converges in $L^2$, first we note that as the second derivative of the bulk energy $\left(\psi_s(\xi)-Q\cdot\xi\right)-\min\limits_{\tilde\xi}\left(\psi_s(\tilde\xi)-k_{DH}^0Q\cdot\tilde\xi\right)$ is positive definite, and independent of $Q$, this implies that there is some $\delta>0$ so that if $|\tilde{\xi}-\xi(x)|<\delta$, then $\left(\psi_s(\xi)-Q\cdot\xi\right)-\min\limits_{\tilde\xi}\left(\psi_s(\tilde\xi)-k_{DH}^0Q\cdot\tilde\xi\right)\geq C|\tilde{\xi}-\xi(x)|^2$. The constant $C$ can be taken uniformly in $Q\in\mm$. So we can estimate
\begin{equation}
\begin{split}
&||\xi_j-\xi||_2^2\\
=& \int_{|\xi_j-\xi|>\delta} |\xi_j(x)-\xi(x)|^2\,dx + \int_{|\xi_j-\xi|\leq \delta} |\xi_j(x)-\xi(x)|^2\,dx .
\end{split}
\end{equation}
To estimate the first term, it suffices to show that the set where $|\xi_j-\xi|>\delta$ is of vanishing measure as $j\to\infty$, as $\xi,\xi_j$ are uniformly bounded. However, the uniform estimate on the bulk energy means that there is some constant $C_1>0$ with 
\begin{equation}
\min\limits_{|\hat\xi-\xi(x)|>\delta}\left(\psi_s(\hat\xi)-Q\cdot\hat\xi\right)-\min\limits_{\tilde\xi}\left(\psi_s(\tilde\xi)-k_{DH}^0Q\cdot\tilde\xi\right)>C_1,
\end{equation}
so this follows from the energy estimate. 

For the second term, we estimate this as 
\begin{equation}
\begin{split}
&\int_{|\xi_j-\xi|\leq \delta} |\xi_j(x)-\xi(x)|^2\,dx \\
\leq & \frac{1}{C}\int_{|\xi_j-\xi|\leq \delta}\left(\psi_s(\xi)-Q\cdot\xi\right)-\min\limits_{\tilde\xi}\left(\psi_s(\tilde\xi)-k_{DH}^0Q\cdot\tilde\xi\right)\,dx\to 0 
\end{split}
\end{equation}
Combining these gives that $||\xi_j-\xi||_2\to 0$. 
\end{proof}

\subsection{Liminf inequality}\label{subsecLiminf}

Within this section we will proceed to show the necessary liminf inequality for our $\Gamma$-convergence result. The necessary ingredients are to show that given an energy estimate, the Host-Host interaction term has the asymptotic lower bound of the quadratic energy in $\nabla Q$, the chiral Host-Dopant term converges to the correct term linear in $\nabla Q$, and all other interaction terms converge to zero, which are strongly aided by our compactness results.

\begin{proposition}\label{propDHToZero}
Let $Q_\e,\xi_\e\in L^\infty(\mt^3,\mqc)$. Assume that $\mf_\epsilon(Q_\epsilon,\xi_\epsilon)$ is uniformly bounded. Then, we may take a subsequence $\epsilon_j\to 0$ so that 
\begin{equation}
\frac{1}{\e_j}\intt\intt K_{DH}^{\epsilon_j} (x-y)\big(Q_{j}(x)-Q_j(y)\big)\cdot(\xi_j(x)-\xi_j(y)\big)\,dy\,dx\to 0.
\end{equation}
\end{proposition}
\begin{proof}
First, we take a subsequence so that $Q_j\to Q\in L^2$, $\xi_j \to \xi=\frac{s_c}{s_0}Q$ in $L^2$. Then estimate 
\begin{equation}
\begin{split}
&\frac{1}{\e_j}\intt\intt K_{DH}^{\epsilon_j} (x-y)\big(Q_{j}(x)-Q_j(y)\big)\cdot(\xi_j(x)-\xi_j(y)\big)\,dy\,dx\\
\leq & \frac{C}{\e_j}\intt\intt g^{\epsilon_j} (x-y)\big|Q_{j}(x)-Q_j(y)\big|\cdot|\xi_j(x)-\xi_j(y)\big|\,dy\,dx\\
= & \frac{C}{\e_j}\intt\intt g^{\epsilon_j} (x-y)\big|Q_{j}(x)-Q_j(y)\big|\cdot\left|\xi_j(x)-\frac{s_c}{s_0}Q_j(x)+\frac{s_c}{s_0}Q_j(x)-\frac{s_c}{s_0}Q_j(y)+\frac{s_c}{s_0}Q_j(y)-\xi_j(y)\right|\,dy\,dx\\
\leq & \frac{C}{\e_j}\intt\intt g^{\epsilon_j} (x-y)\big|Q_{j}(x)-Q_j(y)\big|\left(\frac{s_0}{s_c}\big|Q_j(x)-Q_j(y)\big|+\left|\frac{s_c}{s_0}Q_j(x)-\xi_j(x)\right|+\left|\frac{s_c}{s_0}Q_j(y)-\xi_j(y)\right|\right)\,dy\,dx\\
=&  \frac{C}{\e_j}\intt \intt g^{\epsilon_j} (x-y)\big|Q_{j}(x)-Q_j(y)\big|\left(\frac{s_0}{s_c}\big|Q_j(x)-Q_j(y)\big|+2\left|\frac{s_c}{s_0}Q_j(x)-\xi_j(x)\right|\right)\,dy\,dx.\\
\end{split}
\end{equation} 
we estimate the two summands in turn. For the left-hand term,
\begin{equation}
\begin{split}
 &\frac{C}{\e_j}\intt\intt g^{\epsilon_j} (x-y)\big|Q_{j}(x)-Q_j(y)\big|\frac{s_0}{s_c}\big|Q_j(x)-Q_j(y)\big|\,dy\,dx\\
 =& \epsilon_j \frac{Cs_c}{s_0\e_j^2}\intt\intt g^{\epsilon_j} (x-y)\big|Q_{j}(x)-Q_j(y)\big|^2\,dy\,dx\\
 \leq & C\epsilon_j\left(\mf_{\e_j}(Q_j,\xi_j)+1\right)\to 0
\end{split}
\end{equation}
from the energy bound. For the remaining term, 
\begin{equation}
\begin{split}
& \frac{2C}{\e_j}\intt\intt g^{\epsilon_j} (x-y)\big|Q_{j}(x)-Q_j(y)\big|\left|\frac{s_c}{s_0}Q_j(x)-\xi_j(x)\right|\,dy\,dx\\
\leq & 2C\left(\frac{1}{\e_j^2}\intt\intt g^{\epsilon_j} (x-y)\big|Q_{j}(x)-Q_j(y)\big|^2\,dx\,dy\right)^\frac{1}{2}\left(\intt\intt g^{\e_j}(x-y)\left|\frac{s_c}{s_0}Q_j(x)-\xi_j(x)\right|^2\,dy\,dx\right)^\frac{1}{2}\\
\leq & C\left(\mf_{\e_j}(Q_j,\xi_j)+1\right)\left(\intt\intt g^{\e_j}(x-y)\,dy\left|\frac{s_c}{s_0}Q_j(x)-\xi_j(x)\right|^2\,dx\right)^\frac{1}{2}\\
= & C\left(\mf_{\e_j}(Q_j,\xi_j)+1\right)^\frac{1}{2}||g^{\e_j}||_1^\frac{1}{2}\left|\left|\frac{s_c}{s_0}Q_j-\xi_j\right|\right|_2.
\end{split}
\end{equation}
As the energy is bounded, $g^{\e_j}$ is $L^1$-bounded, and $\xi_j,\frac{s_c}{s_0}Q_j$ both converge to $\xi$ in $L^2$, this term also vanishes as $\epsilon\to 0$, giving the required result. 
\end{proof}

\begin{proposition}\label{propCDTo0}
Let $Q_\e,\xi_\e\in L^\infty(\mt^3,\mqc)$. Assume that $\mf_\e(Q_\e,\xi_\e)$ is uniformly bounded. Then 
\begin{equation}
\intt\intt K_{cD}^\e(x-y)\xi(x)\cdot (\xi(x)-\xi(y))\,dy\,dx\to 0.
\end{equation}
\end{proposition}
\begin{proof}
We estimate 
\begin{equation}
\begin{split}
&\left|\intt\intt K_{cD}^\e(x-y)\xi(x)\cdot (\xi(x)-\xi(y))\,dy\,dx\right|\\
\leq & C\epsilon\intt\intt \frac{1}{\epsilon}g^\e(x-y)|\xi_\e(x)-\xi_\e(y)|\,dx\,dy\\
\leq & C\epsilon\left(\intt\intt \frac{1}{\epsilon^2}g^\e(x-y)|\xi_\e(x)-\xi_\e(y)|^2\,dx\,dy\right)^\frac{1}{2}\left(\intt\intt g^\e(x-y)\,dx\,dy\right)^\frac{1}{2}\\
= & C\epsilon||g^{\epsilon}||_1^\frac{1}{2}\left(\intt\intt \frac{1}{\epsilon^2}g^\e(x-y)|\xi_\e(x)-\xi_\e(y)|^2\,dx\,dy\right)^\frac{1}{2}\\
\leq & C\epsilon||g^{\epsilon}||_1^\frac{1}{2}(\mf_\e(Q_\e,\xi_\e)+1)^\frac{1}{2}=O(\epsilon)
\end{split}
\end{equation}
from the energy estimate
\end{proof}

Before proceeding with the next result we introduce some further notation for finite difference operators.
\begin{definition}
Let $u:\mathbb{R}^3\to V$ for a vector space $V$. Let $z\in\mathbb{R}^3$. Define the finite difference $D_zu:\mathbb{R}^3\to V$ by 
\begin{equation}
D_zu(x)=\frac{1}{|z|}\big(u(x+z)-u(x)\big).
\end{equation}
\end{definition}

\begin{proposition}\label{propChiralTermConverge}
Assume that $Q_\e,\xi_\e$ converge in $L^2$ to $Q,\xi\in W^{1,2}(\mathbb{T}^3,\mqc)$ with $\mf_\epsilon(Q_\epsilon,\xi_\epsilon)<M$. Then we have a subsequence $Q_j=Q_{\e_j}$,$\xi_j=\xi_{\e_j}$ so that 
\begin{equation}
\frac{1}{\epsilon}\intt \intt K_{cH}^\epsilon(x-y)\xi_j(x)\cdot (Q_j(x)-Q_j(y))\,dx\,dy\\
\to \intt\xi(x)\cdot  \left(\int_{\mathbb{R}^3}K_{cH}(z)\otimes z\,dz\right)\cdot \nabla Q(x)\,dx.
\end{equation}
\end{proposition}
\begin{proof}
First, we define $u_\epsilon(x)=\frac{1}{\epsilon}\intt K^\epsilon_{cH}(x-y)\big(Q_\epsilon(x)-Q_\epsilon(y)\big)\,dy$. The strategy is first to obtain a uniform $L^2$ bound on $u_\e$ so that we may take a weakly converging subsequence, and then the limit can be identified by investigating the limit in the sense of distributions. We estimate
\begin{equation}
\begin{split}
||u_\epsilon||_2^2 =& \frac{1}{\epsilon^2}\intt\left|\intt K_{cH}^\epsilon(x-y)\big(Q_\epsilon(x)-Q_\epsilon(y)\big)\,dy\right|^2\,dx\\
\leq & \frac{C}{\epsilon^2} \intt\left(\intt g^\epsilon(x-y)\big|Q_\epsilon(x)-Q_\epsilon(y)\big|\,dy\right)^2\,dx\\
\leq &  \frac{C}{\epsilon^2} \intt\left(\intt g^\epsilon(x-y)\big|Q_\epsilon(x)-Q_\epsilon(y)\big|^2\,dy\right)\left(\intt g^\epsilon(x-y)\,dy\right)\,dx\\
\leq &  \frac{C}{\epsilon^2} \intt\intt g^\epsilon(x-y)\big|Q_\epsilon(x)-Q_\epsilon(y)\big|^2\,dy\,dx\\
\leq & C\left(\mf_\epsilon(Q_\epsilon,\xi_\epsilon)+1\right)
\end{split}
\end{equation}
So this implies that $u_\epsilon$ is $L^2$ bounded. Hence we can take a subsequence $\epsilon_j$ with $Q_j \overset{L^2}{\to} Q \in W^{1,2}(\mt^3,\mm)$ and $Q_j(x)\to Q(x)$ pointwise almost everywhere, ${\xi_j=\xi_{\epsilon_j}\overset{L^2}{\to}\frac{s_c}{s_0}Q}$, and $u_{\epsilon_j}=u_j\overset{L^2}{\rightharpoonup} u$. We find $u$ by integrating against a test function $\phi \in \mathcal{D}(\mt^3)$. This gives, by a change of variables $z=x-y$, 
\begin{equation}
\begin{split}
\intt u_j(x)\phi(x)\,dx =& \frac{1}{\epsilon}\intt\intt\phi(x)K_{cH}^{\epsilon_j}(x-y)\big(Q_j(x)-Q_j(y)\big)\,dy\,dx\\
=& -\frac{1}{\epsilon}\intt\intt\phi(x)K_{cH}^{\epsilon_j}(z)|z|D_{-z}Q_j(x)\,dz\,dx\\
=&  -\frac{1}{\epsilon}\intt K_{cH}^{\epsilon_j}(z)|z|\intt\phi(x)D_{-z}Q_j(x)\,dx\,dz\\
=& -\frac{1}{\epsilon}\intt K_{cH}^{\epsilon_j}(z)|z|\intt D_z\phi(x)Q_j(x)\,dx\,dz.
\end{split}
\end{equation}
Now, as $\phi$ is smooth and periodic,
\begin{equation}
\frac{1}{\epsilon}\intt\intt K_{cH}^{\epsilon_j}(z)Q_j(x)|z| D_z\phi(x)\,dz\,dx=\intt\int_{\mathbb{R}^3}K_{cH}(z)Q_j(x)|z|D_{\epsilon z}\phi(x)\,dz\,dx.
\end{equation}
The integrand can be bounded by $Cg(z)|z|||\nabla\phi||_\infty$ pointwise almost everywhere, which, as $g\in L^1$ with finite second moment and $\mt^3$ is bounded, is integrable over $(x,z)\in\mt^3\times\mathbb{R}^3$. Furthermore, as $\phi$ is smooth, $D_{\epsilon z}\phi(x)\to(\hat{z}\cdot \nabla)\phi(x)$. Recalling $Q_j(x)$ converges pointwise almost everywhere to $Q(x)$ also, this implies that the integrand converges pointwise to $K_{cH}(z)Q(x)(z\cdot \nabla)\phi(x)$ as $\epsilon \to 0$. Thus by dominated convergence, 
\begin{equation}\begin{split}
&\lim\limits_{\epsilon\to 0} -\frac{1}{\epsilon}\intt K_{cH}^{\epsilon_j}(z)|z|\intt D_z\phi(x)Q_j(x)\,dx\,dz\\
=& -\intt \int_{\mathbb{R}^3}K_{cH}(z)Q(x)(z\cdot\nabla)\phi(x)\,dz\,dx\\
=& \intt \left(\int_{\mathbb{R}^3}K_{cH}(z)\otimes z\,dz\right)\cdot \nabla Q (x)\phi(x)\,dx.
\end{split}
\end{equation}
As the limit in $\mathcal{D}'(\mathbb{T}^3)$ and weak limit in $L^2$ must coincide,
\begin{equation}
u= \left(\int_{\mathbb{R}^3}K_{cH}(z)\otimes z\,dz\right)\cdot \nabla Q.
\end{equation}
Now, finally, we note that 
\begin{equation}\begin{split}
&\frac{1}{\epsilon}\intt \intt K_{cH}^\epsilon(x-y)\xi_j(x)\cdot (Q_j(x)-Q_j(y))\,dx\,dy\\
=& \intt \xi_j(x)\cdot u_j(x)\,dx\\
\to &\intt\xi(x)\cdot u(x)\,dx\\
=&  \intt\xi(x)\cdot \left(\int_{\mathbb{R}^3}K_{cH}(z)\otimes z\,dz\right)\cdot \nabla Q(x)\,dx,
\end{split}
\end{equation}
as $\xi_j$ converges in $L^2$ strongly.
\end{proof}

\begin{theorem}\label{thmLiminf}
Let $ Q \in W^{1,2}(\mt^3,\mm)$, $Q_\epsilon \overset{L^2}{\to} Q$, $\xi_\epsilon\overset{L^2}{\to}\frac{s_c}{s_0}Q$. Then 
\begin{equation}
\liminf\limits_{\epsilon\to 0} \mf_\e(Q_\e,\xi_\e)\geq \intt \frac{1}{2}L\nabla Q\cdot \nabla Q+\frac{s_c\rho_0}{s_0}VQ\cdot \nabla Q\,dx -\frac{s_c^2\rho_0^2k_{DD}^0}{3}|\Omega|
\end{equation}
Here $L$, $V$ are given as 
\begin{equation}
\begin{split}
L\nabla Q\cdot \nabla Q=&\frac{1}{2}\left(\int_{\mathbb{R}^3}K_{HH}(z)z_\alpha z_\beta\,dz\right)\frac{\partial Q}{\partial x_\alpha}\cdot \frac{\partial Q}{\partial x_\beta},\\
VQ\cdot \nabla Q=&\left(\int_{\mathbb{R}^3}K_{cH}(z)z_\alpha\,dz\right)Q\cdot \frac{\partial Q}{\partial x_\alpha}.
\end{split}
\end{equation}
\end{theorem}
\begin{proof}
We assume the liminf is finite, else there is nothing to prove. By observing many of the terms in the energy are non-negative, we have for any $Q_\e,\xi_\e$
\begin{equation}\begin{split}
&\mf_\e(Q_\e,\xi_\e)\\
\geq & \frac{1}{\epsilon^2}\intt\intt \frac{1}{4}K^{\e}_{HH}(x-y)\big(Q_\e(x)-Q_\e(y)\big)^{\otimes 2}+\rho_0\epsilon K_{cH}^\e(x-y)\xi_e(x)\cdot(Q_\e(x)-Q_\e(y))\\
&+\frac{1}{2}\rho_0\epsilon K_{DH}^\e(x-y)(Q_\e(x)-Q_\e(y))\cdot (\xi_\e(x)-\xi_e(y))+\rho_0^2\epsilon^2 K_{cD}^\e(x-y)\xi_e(x)\cdot(\xi_\e(x)-\xi_\e(y))\,dy\,dx\\
&-\frac{\rho_0^2k_{DD}^0}{2}||\xi_\epsilon||^2
\end{split}
\end{equation}
From \Cref{propCDTo0,propDHToZero}, we see that the second line terms vanish as $\epsilon\to 0$. The $L^2$ norm of $\xi_\epsilon$ converges due to the $L^2$ convergence of $\xi_\epsilon$. From \Cref{propChiralTermConverge} it is seen that 
\begin{equation}\begin{split}
\rho_0\epsilon K_{cH}^\e(x-y)\xi_e(x)\cdot(Q_\e(x)-Q_\e(y))\to &\rho_0\intt V\xi\cdot \nabla Q\,dx\\
=&\frac{s_c\rho_0}{s_0}\intt VQ\cdot \nabla Q\,dx\end{split}
\end{equation}
and from \Cref{propLiminfNonlocal} we have that 
\begin{equation}
\liminf\limits_{\epsilon\to 0} \frac{1}{4\epsilon^2}\intt\intt K^{\e}_{HH}(x-y)\big(Q_\e(x)-Q_\e(y)\big)^{\otimes 2}\geq \intt \frac{1}{2}L\nabla Q\cdot\nabla Q\,dx.
\end{equation}
Combining these gives the required result. 
\end{proof}

\subsection{Limsup inequality}\label{subsecLimsup}

The ``bulk"-type contribution to the energy is singular as $\epsilon\to 0$, so in constructing recovery sequences it is essential we can gain strong control over the bulk energy. The simplest means to do so will be to construct recovery sequences that, for each $\epsilon>0$, have zero bulk energy. As the bulk energy itself depends on $\epsilon$ in a non-trivial way, before proceeding we must first understand the behaviour of the minimisers of the bulk energy. The perturbative nature of the $\epsilon$ dependence makes this tractable however.

\begin{proposition}\label{propUniaxialApprox}
There exists $\epsilon^*>0$ so that for all $0<\epsilon<\epsilon^*$, there exists minimisers of ${\psi_s(Q)-\frac{k_{HH}^0}{2}|Q|^2+\epsilon\rho_0(\psi_s(\xi)-k_{HD}^0Q\cdot \xi)}$ of the form 
\begin{equation}\begin{split}
Q_\epsilon=&s_0(\epsilon)\sigma(n),\\
\xi_\epsilon=&s_c(\epsilon)\sigma(n).\end{split}
\end{equation}
Furthermore, as $\epsilon\to 0$, $s_0(\epsilon)\to s_0$, $s_c(\epsilon)\to s_c$. 
\end{proposition}
\begin{proof}
Let $\epsilon>0$. First, we minimise with respect to $\xi$ in the same manner as before, yielding $\Lambda(\xi)=k_{HD}^0Q$. We substitute this back into the expression, recalling the alternative expression for $\psi_s$, as 
\begin{equation}
\begin{split}
&\psi_s(\xi)-k_{HD}^0Q\cdot \xi\\
=&  - \ln\int_{\mathbb{S}^2}\exp(k_{HD}^0Qp\cdot p)\,dp.
\end{split}
\end{equation}
Returning to our minimisation problem, we have 
\begin{equation}
\begin{split}
\min\limits_{Q,\xi}=& \psi_s(Q)-\frac{k_{HH}^0}{2}|Q|^2+\epsilon\rho_0(\psi_s(\xi)-k_{HD}^0Q\cdot \xi)\\
=&\min\limits_{Q}\psi_s(Q)-\rho_0\frac{k_{HH}^0}{2}|Q|^2-\epsilon\rho_0\ln\int_{\mathbb{S}^2}\exp(k_{HD}^0Qp\cdot p)\,dp.
\end{split}
\end{equation}
We proceed by a perturbation argument. Define the perturbed bulk energy as \begin{equation}
\psi_B^\e(Q)=\psi_s(Q)-\frac{k_{HH}^0}{2}|Q|^2-\epsilon\rho_0\ln\int_{\mathbb{S}^2}\exp(k_{HD}^0Qp\cdot p)\,dp.
\end{equation} 
First, we wish to remove the degeneracy of symmetry from the problem, as $\psi_B^\e$ is frame indifferent. We consider Q-tensors to be parametrised by $\eta(s,t)=s\sigma(e_1)+t(e_2\otimes e_2-e_3\otimes e_3)$, for $s,t\in\mathbb{R}$. The orbit of such Q-tensors under $\text{SO}(3)$ gives all of $\mathcal{Q}$. In this coordinate system, fixing the orthogonal basis $e_1,e_2,e_3$ and treating $s,t$ as parameters, the global minimisers of $\psi_B^0$ are isolated in $(s,t)$-space. All global minimisers of $\psi_B^0$ are equivalent under $\text{SO}(3)$, so we investigate the most convenient, $\eta(s_0,0)$. As $\psi_B^\e$ is smooth in $\e$,to show minimisers of $\psi_B^\e$ are uniaxial for small $\e$, it suffices to show no branch of biaxial solutions $Q_{\e}=\eta(s_\e,t_\e)$ intersects $Q_0$ at $\e=0$, where biaxiality implies $t_\e\neq 0$. As $\eta(s,t)$ and $\eta(s,-t)$ are equivalent under rotations, if solutions were to be biaxial for all $\epsilon>0$, then the branch would split into the two symmetry related configurations. However, we know that there is no bifurcation due to the stability result of Li {\it et. al.} \cite{li2014local}, which states $\frac{\partial^2}{\partial (s,t)^2}\psi_B^0(Q_0)>0$. As a consequence, there must be a continuous branch of solutions with $t=0$ in $(s,t)$ coordinates, so we can write them as $Q_\epsilon=\eta(s_0(\epsilon),0)=s_0(\epsilon)\sigma(e_1)$. In particular, the continuity implies $s_0(\epsilon)\to s_0$.

Recalling that $\xi_\epsilon$ satisfies the equation $\Lambda(\xi_\epsilon)=k_{HD}^0Q_\epsilon$, we can use that $\Lambda$ is frame indifferent to infer that the rotations with $R\xi_\epsilon R^T=\xi_\epsilon$ are precisely those with $RQ_\epsilon R^T=Q_\epsilon$, hence $\xi_\epsilon$ must be uniaxial with the same director as $Q_\epsilon$, hence $\xi_\epsilon=s_c(\epsilon)\sigma(e_1)$. Furthermore, as $\Lambda$ admits a continuous inverse, and $Q_\epsilon\to Q$, this implies \begin{equation}
s_c(\epsilon)\sigma(e_1)=\xi_\epsilon=\Lambda^{-1}(k_{HD}^0Q_\epsilon)\to \Lambda^{-1}(k_{HD}^0Q_0)=\xi_0=s_c\sigma(e_1).
\end{equation} 
Hence $s_c(\epsilon)\to s_c$. 
\end{proof}

\begin{proposition}
Let $Q\in W^{1,2}(\mathbb{T}^3,\mm)$. Let $Q_\epsilon=\frac{s_0(\epsilon)}{s_0}Q$, $\xi_\epsilon = \frac{s_c(\epsilon)}{s_0}Q$. Then $Q_\epsilon\to Q$ in $W^{1,2}$, $\xi_\epsilon \to \frac{s_c}{s_0}Q$, and \begin{equation}
\lim\limits_{\epsilon\to 0} \mf_\epsilon(Q_\epsilon,\xi_\epsilon)= \intt \frac{1}{2}L\nabla Q\cdot \nabla Q+\frac{s_c\rho_0}{s_0}VQ\cdot \nabla Q\,dx -\frac{s_c^2\rho_0^2k_{DD}^0}{3}|\Omega|.
\end{equation}
\end{proposition}
\begin{proof}
The fact that $Q_\epsilon,\xi_\epsilon\overset{W^{1,2}}{\to}Q_0,\xi_0$ is immediate from \Cref{propUniaxialApprox}. It remains to show the energy admits the correct limit. As $Q_\epsilon,\xi_\epsilon$ by definition have zero bulk energy, and by \Cref{propCDTo0,propDHToZero}, we can remove redundant terms in the energy that vanish in the limit to give
\begin{equation}\begin{split}
&\lim\limits_{\epsilon\to 0} \mf_\epsilon(Q_\epsilon,\xi_\epsilon)\\
=&\lim\limits_{\epsilon\to 0} \intt\intt \frac{1}{4\epsilon^2}K_{HH}^\epsilon(x-y)\left(Q_\epsilon(x)-Q_\epsilon(y)\right)^{\otimes 2} +\frac{1}{\epsilon}K^\epsilon_{cH}\xi_\epsilon(x)\cdot (Q_\epsilon(x)-Q_\epsilon(y))\,dy\,dx-\intt\frac{k_{DD}^0}{2}|\xi_\epsilon(x)|^2\,dx\\
=&\lim\limits_{\epsilon\to 0} \intt\intt \frac{s_0(\epsilon)^2}{4s_0^2\epsilon^2}K_{HH}^\epsilon(x-y)\left(Q(x)-Q(y)\right)^{\otimes 2} +\frac{s_c(\epsilon)s_0(\epsilon)}{s_0^2\epsilon}K^\epsilon_{cH}Q(x)\cdot (Q(x)-Q(y))\,dy\,dx-\intt\frac{s_c(\epsilon)^2k_{DD}^0}{2s_0^2}|Q(x)|^2\,dx.
\end{split}
\end{equation}
Now, we use the convergence result of \Cref{propLimsupNonlocal} on the first term, \Cref{propChiralTermConverge} on the second term, and and that $s_0(\epsilon)\to s_0$, $s_c(\epsilon)\to s_c$, to give 
\begin{equation}
\lim\limits_{\epsilon\to 0} \mf_\epsilon(Q_\epsilon,\xi_\epsilon)=\intt \frac{1}{2}L\nabla Q\cdot \nabla Q+\frac{s_c\rho_0}{s_0}VQ\cdot \nabla Q\,dx -\frac{s_c^2\rho_0^2k_{DD}^0}{3}|\Omega|.
\end{equation}
\end{proof}

This now leads us to the main result.
\begin{theorem}\label{thmPeriodicGamma}
The functionals $\mf_\e:L^\infty(\mt^3;\mqc)^2\to\mathbb{R}\cup\{+\infty\}$ as defined in \eqref{eqEnergy} $\Gamma$-converge with respect to strong-$L^2$ convergence to the functional $\mf_{OF}:L^\infty(\mt^3;\mqc^2)\to\mathbb{R}\cup\{+\infty\}$ defined by 
\begin{equation}
\mf_{OF}(Q,\xi)=\intt \frac{1}{2}L\nabla Q\cdot \nabla Q+\frac{s_c\rho_0}{s_0}VQ\cdot\nabla Q\,dx-\frac{s_c^2\rho_0^2k_{DD}^0}{3}|\Omega|
\end{equation}
if $Q,\xi\in W^{1,2}(\mt^3,\mqc^2)$, and $Q(x)\in\mm$ and $\xi(x)=\frac{s_c}{s_0}Q(x)$ pointwise almost everywhere, and $+\infty$ otherwise. The operators $L,V$ are defined by 
\begin{equation}
\begin{split}
L\nabla Q\cdot \nabla Q =&\frac{1}{2}\left(\int_{\mathbb{R}^3}K_{HH}(z)z_\alpha z_\beta\,dz\right)\frac{\partial Q}{\partial x_\alpha}\cdot \frac{\partial Q}{\partial x_\beta},\\
V Q\cdot \nabla Q =&\left(\int_{\mathbb{R}^3}K_{cH}(z)z_\alpha\,dz\right)Q\cdot \frac{\partial Q}{\partial x_\alpha}
\end{split}
\end{equation}
\end{theorem}

\section{Coefficients in standard form}\label{sectionCoefficients}
\subsection{Landau coefficients}
In $\mf_{OF}$, the term quadratic in the $\nabla Q$ is given by $\frac{1}{2}L\nabla Q\cdot\nabla Q$, and is frame indifferent. This reduces the degrees of freedom of $L$ down to three elastic constants, and by symmetry we know that any such function must be of the form \begin{equation}
L\nabla Q\cdot \nabla Q=L_1Q_{ij,j}Q_{ik,k}+L_2Q_{ij,k}Q_{ij,k}+L_3Q_{ij,k}Q_{ik,j}
\end{equation}
for scalars $L_1,L_2,L_3$ that can be obtained by testing various configurations of $\nabla Q$ (see e.g. \cite{mori1999multidimensional}). Relating the constants $L_1,L_2,L_3$ to various integrals of $K_{HH}$ has been done previously in \cite{taylor2018oseen} within the context of this model, to which we reference the reader for explicit formulae.

We now perform a similar symmetry argument here for the novel term, linear in $\nabla Q$. Recall that for $W_{ij}=\epsilon_{i\alpha j}z_\alpha$, we can write 
\begin{equation}
K_{cX}(z)A\cdot B=f_1(z)\text{Tr}(AWB)+f_2(z)\text{Tr}(WA(\hat{z}\otimes\hat{z}-I)B),
\end{equation}
with $f_1,f_2$ isotropic functions of $z$ (see \Cref{appSymmetry}).

The term in the energy linear in $\nabla Q$ is given by 
\begin{equation}
VQ\cdot\nabla Q = \int_{\mathbb{R}^3}f_1(z)\text{Tr}(QW(z\cdot\nabla)Q)+f_2(z)\text{Tr}(WQA(\hat{z}\otimes\hat{z}-I)(z\cdot\nabla)Q)
\end{equation}

Taking $A=Q$, $B=(z\cdot \nabla)Q$, integrating each part in turn we have 
\begin{equation}
\begin{split}
&\int_{\mathbb{R}^3}f_1(z)\text{Tr}(QW(z\cdot \nabla Q))\,dz\\
=& \int_{\mathbb{R}^3}f_1(z)Q_{ij}W_{jk}z_lQ_{ki,l}\,dz\\
=& \left(\int_{\mathbb{R}^3}f_1(z)z_m z_l\,dz\right)\epsilon_{jmk}Q_{ij}Q_{ki,l}\\
=& \beta_1\epsilon_{jlk}Q_{ij}Q_{ki,l}.
\end{split}
\end{equation}
The scalar $\beta_1$ is given by 
\begin{equation}
\beta_1=\frac{1}{3}\int_{\mathbb{R}^3}f_1(z)|z|^2\,dz=\frac{4\pi}{3}\int_0^\infty f_1(re_1)r^4\,dr.
\end{equation}

The fact that such a scalar representation exists follows from the isotropy of $f_1$, which implies $\int_{\mathbb{R}^3}z\otimes zf(z)\,dz$ must be a multiple of the identity. 

Turning to the second term, 
\begin{equation}
\begin{split}
&\int_{\mathbb{R}^3}f_2(z)\text{Tr}(WQ(\hat{z}\otimes \hat{z}-I)(z\cdot \nabla Q))\,dz\\
=&\int_{\mathbb{R}^3}f_2(z)W_{ij}Q_{jk}(\hat{z}_k\hat{z}_l-\delta_{kl})z_m  Q_{li,m}\,dz\\
=&\left(\int_{\mathbb{R}^3}f_2(z)z_n(\hat{z}_k\hat{z}_l-\delta_{kl})z_m  \,dz \right)\epsilon_{inj}Q_{jk}Q_{li,m}\\
=&\left(\int_{\mathbb{R}^3}f_2(z)z_n\hat{z}_k\hat{z}_lz_m-f_2(z)z_n\delta_{kl}z_m  \,dz \right)\epsilon_{inj}Q_{jk}Q_{li,m}
\end{split}
\end{equation}

To deal with the first term in the integrand, we exploit that $\frac{1}{4\pi}\int_{\mathbb{S}^2}p_ip_jp_kp_l\,dp=\frac{1}{15}\left(\delta_{ij}\delta_{kl}+\delta_{il}\delta_{jk}+\delta_{ik}\delta_{jl}\right)$ \cite{han2015microscopic}. This gives 
\begin{equation}
\begin{split}
&\int_{\mathbb{R}^3}f_2(z)z_n\hat{z}_k\hat{z}_lz_m\,dz\\
=&(4\pi)\int_0^\infty f_2(re_1)r^6\,dr\frac{1}{4\pi}\int_{\mathbb{S}^2}p_np_kp_lp_m\,dp\\
=& \beta_2\left(\delta_{nk}\delta_{lm}+\delta_{nl}\delta_{km}+\delta_{nm}\delta_{kl}\right),
\end{split}
\end{equation}
with
\begin{equation}
\beta_2=\frac{1}{15}(4\pi)\int_0^\infty f_2(re_1)r^6\,dr.
\end{equation}
For the second term, we have by a similar argument to before, 
\begin{equation}
\int_{\mathbb{R}^3}f_2(z)z_n\delta_{kl}z_m  \,dz=\beta_3 \delta_{nm}\delta_{kl}
\end{equation}
with 
\begin{equation}
\beta_3=\frac{4\pi}{3}\int_0^\infty f_3(re_1)r^4\,dr. 
\end{equation}

Thus we return to the full expression, to give 
\begin{equation}
\begin{split}
&\int_{\mathbb{R}^3}f_2(z)\text{Tr}(WQ(\hat{z}\otimes \hat{z}-I)(z\cdot \nabla Q))\,dz\\
=&\bigg( \beta_2\left(\delta_{nk}\delta_{lm}+\delta_{nl}\delta_{km}+\delta_{nm}\delta_{kl}\right)-\beta_3 \delta_{nm}\delta_{kl}\bigg)\epsilon_{inj}Q_{jk}Q_{li,m}\\
=&\beta_2\epsilon_{ikj}Q_{jk}Q_{li,l}+\beta_2\epsilon_{inj}Q_{jk}Q_{ni,k} +\beta_2\epsilon_{inj}Q_{jk}Q_{ki,n}-\beta_3\epsilon_{inj}Q_{jk}Q_{ki,n}\\
=&\beta_2\epsilon_{ikj}Q_{jk}Q_{li,l}+\beta_2\epsilon_{inj}Q_{jk}Q_{ni,k} +\beta_2\epsilon_{inj}Q_{jk}Q_{ki,n}-\beta_3\epsilon_{inj}Q_{jk}Q_{ki,n}\\
\end{split}
\end{equation}
Simplifying some of the terms, since $Q_{jk}=Q_{kj}$, but $\epsilon_{ikj}=-\epsilon_{ijk}$, this implies $\epsilon_{ikj}Q_{jk}=0$. Similarly, $\epsilon_{inj}Q_{ni,k}=0$, leaving only two terms, 
\begin{equation}
\left(\beta_2-\beta_3\right)\epsilon_{inj}Q_{jk}Q_{ki,n}
\end{equation}

All in all, this gives 
\begin{equation}
VQ\cdot \nabla Q = \big(\beta_1+\beta_2-\beta_3)\epsilon_{jkl}Q_{ij}Q_{ik,l}=\beta \epsilon_{jkl}Q_{ij}Q_{ik,l},.
\end{equation}
with 
\begin{equation}
\beta = \frac{4\pi}{3}\int_0^\infty f_1(re_1)r^4+\frac{1}{5}f_2(re_1)r^6-f_3(re_1)r^4\,dr. 
\end{equation}

\subsection{Frank coefficients and HTP}\label{subsecFrankCoeffs}
Assume that $Q$ is (at least locally) orientable, so it can be written as $Q=s_0\sigma(n)$ for some $n\in W^{1,2}(\Omega,\mathbb{S}^2)$. Then the term linear in the gradient of $Q$ becomes 
\begin{equation}
\begin{split}
VQ\cdot \nabla Q =& \beta\epsilon_{jkl}Q_{ij}Q_{ik,l}\\
=& s_0^2\beta\epsilon_{jkl}\left(n_in_j-\frac{1}{3}\delta_{ij}\right)\big(n_in_{k,l}+n_{i,l}n_k\big)\\
=&s_0^2\beta\left(\epsilon_{jkl}n_jn_{k,l}-\frac{1}{3}\epsilon_{jkl}n_jn_{k,l}-\frac{1}{3}\epsilon_{jkl}n_{j,l}n_k \right)\\
=&\beta s_0^2n\cdot\nabla\times n
\end{split}
\end{equation}

Combining these, and discarding an irrelevant additive constant, in the case of orientable configurations we have the $\Gamma$-limit as 
\begin{equation}\begin{split}
\mf_{OF}(n)=&s_0^2\intt \frac{1}{2}K_{11}(\nabla \cdot n)^2+\frac{1}{2}K_{22}(n\cdot \nabla \times n )^2+\frac{1}{2}K_{33}|n\times \nabla\times n|^2+\frac{s_c\beta \rho_0}{s_0}n\cdot \nabla \times n\,dx\\
=&s_0^2\intt \frac{1}{2}K_{11}(\nabla \cdot n)^2+\frac{1}{2}K_{22}(n\cdot \nabla \times n +q)^2+\frac{1}{2}K_{33}|n\times \nabla\times n|^2\,dx +C,
\end{split}
\end{equation}
where the ground state wavenumber $q$ is given by 
\begin{equation}\label{eqWavenumber}
q=\frac{s_c\beta\rho_0}{s_0K_{22}},
\end{equation}
and the helical pitch is given by $\frac{\pi}{q}$. Now we recall that this all takes place within a domain rescaled according to a small parameter $\epsilon$, and that constants have been non-dimensionalised. We now rephrase \eqref{eqWavenumber} into our system's original units. 

$K_{22}$ can be written as $\frac{\rho_H}{k_BT}\tilde{K_{22}}$ where $\tilde{K}_{22}$ is a constant with units of energy depending only on the hosts structure itself, and not its temperature or concentration. Similarly, we can write $\beta=\frac{\rho_H}{k_BT}\tilde{\beta}$, where $\tilde{\beta}$ is a temperature and concentration independent quantity, depending only on the species of host and dopant. Furthermore, we recall in our scaling that lengths had been non-dimensionalised, so if the wavenumber of our system is $q$ in dimensionless units, it is $\tilde{q}=\frac{q}{\epsilon}$ in our original units. Similarly, the true concentration of dopant is $\rho_D=\rho_0\rho_H\epsilon$. Hence reverting \eqref{eqWavenumber} into the original units of our problem, we have 
\begin{equation}\begin{split}
\frac{1}{\epsilon} \tilde{q}=&\frac{s_c\frac{\rho_H}{k_BT}\tilde{\beta}\rho_H\rho_D}{\rho_H\epsilon s_0\frac{\rho_H}{k_BT}\tilde{K}_{22}},\\
\Rightarrow \tilde{q} =& \frac{s_c\tilde{\beta}\rho_D}{\rho_Hs_0\tilde{K}_{22}}
\end{split}
\end{equation}

 Thus this gives our expression for the helical twisting power in dilute systems, as 
\begin{equation}
h=\left.\frac{\partial q}{\partial\rho_D}\right|_{\rho_D=0}=\frac{s_c\tilde{\beta}}{\rho_Hs_0\tilde{K}_{22}}
\end{equation} 
Again, we emphasise that $\tilde{\beta},\tilde{K}_{22}$ are temperature and concentration independent, defined entirely in terms of the pairwise Host-Dopant and Host-Host interactions, respectively. The temperature dependence of $h$ is thus encoded entirely within the ratio of the order parameters $s_c, s_0$, and generally we expect this to be a highly nonlinear description. However as $s_0,s_c$
solve one-dimensional minimisation problems, it is an elementary exercise to numerically illustrate the temperature dependence. The value of $s_0$ is determined entirely by $k^0_{HH}$, which can be written as $k^0_{HH}=\frac{\rho_H\tilde{k}_1}{k_BT}$ for a temperature and concentration independent constant $\tilde{k}_1$. There are explicit formulae for the relationship between $s_0$ and $k^0_{HH}$, available in \cite{fatkullin2005critical}. The order parameter $s_c$ is determined entirely by $k^0_{HD}=\frac{\rho_H\tilde{k}_2}{k_BT}$, where again $\tilde{k}_2$ is temperature and concentration dependent. $s_c$ can be explicitly obtained from $s_0$ as 
\begin{equation}
s_c=\frac{\int_{-1}^1 P_2(x)\exp\big(k_{HD}^0s_0P_2(x)\big)\,dx}{\int_{-1}^1\exp\big(k_{HD}^0s_0P_2(x)\big)\,dx},
\end{equation} 
where $P_2$ denotes the second Legendre polynomial, which follows from \eqref{eqXiCritical} and \eqref{eqLambdaMinus1}. We define a rescaled temperature $\tau=\frac{1}{k^0_{HH}}=\frac{k_BT}{\rho_H\tilde{k}_1}$, and parameter $\alpha=\frac{\tilde{k}_2}{\tilde{k}_1}=\frac{k^0_{HD}}{k^0_{HH}}$, which satisfies $\frac{\alpha}{\tau}=\rho_H\frac{\tilde{k}_2}{k_BT}$. From these we present the dependence of $h$, or more appropriately the dimensionless quantity $\tilde{h}=\frac{\tilde{K}_{22}\rho_H}{\tilde{\beta}}h=\frac{s_c}{s_0}$, in terms of $\tau$ for representative values of $\alpha$. We define $\tilde{h}$ as it is the only contribution to the HTP that is explicitly temperature dependent. 

Notably, when $\alpha=1$, $\tilde{h}$ is independent of the temperature. To see this, first we define the scalar function $\lambda:\left(-\frac{1}{2},1\right)\to\mathbb{R}$ by \begin{equation}\label{eqLambdaScalar}
\lambda(s)=\frac{3}{2}\Lambda(s\sigma(n))\cdot\sigma(n),
\end{equation} where by frame invariance the choice of $n\in\mathbb{S}^2$ does not effect the definition. Using this notation, the dependence of $s_0,s_c$ on material constants and parameters can be seen through their critical point conditions, 
\begin{equation}
\begin{split}
0=& \lambda(s_0)-\frac{s_0}{\tau},\\
0=& \lambda(s_c)-\alpha\frac{s_0}{\tau}.
\end{split}
\end{equation}
Thus when $\alpha=1$, $\lambda(s_c)=\lambda(s_0)$ so $s_c=s_0$ and $\tilde{h}=1$. Recalling the definition of $\alpha$, $\alpha=1$ corresponds to the case where $k^0_{HH}=k^0_{DH}$. That is, the (average) Host-Host interaction is equal to that of the Host-Dopant interaction. 

Moreso, we can write \begin{equation}\tilde{h}=\frac{s_c}{s_0}=\frac{\lambda^{-1}\left(\frac{\alpha s_0}{\tau}\right)}{\lambda^{-1}\left(\frac{s_0}{\tau}\right)}.\end{equation}
As $\lambda^{-1}$ is strictly increasing, this implies that if $\alpha<1$, $\tilde{h}<1$, and if $\alpha>1$, $\tilde{h}>1$. We also see that as $\tau \to 0$, as $s_c,s_0\to 1$, this means $\tilde{h}\to 1$. We also see that as $\alpha\to 0$, $\tilde{h}$ must be approximately linear in $\alpha$, and as $\alpha\to\infty$, $s_c\to 1$ so $\tilde{h}\to\frac{1}{s_0}$. We show the numerically found results below, for $\alpha=0.15$ showing a decreasing, convex dependence on $\tau$, $\alpha=0.75$ showing a concave, decreasing dependence on $\tau$, and $\alpha=2$, showing a convex increasing dependence on $\tau$. 
\begin{figure}
\begin{subfigure}[b]{0.3\textwidth}
\includegraphics[width=0.9\textwidth]{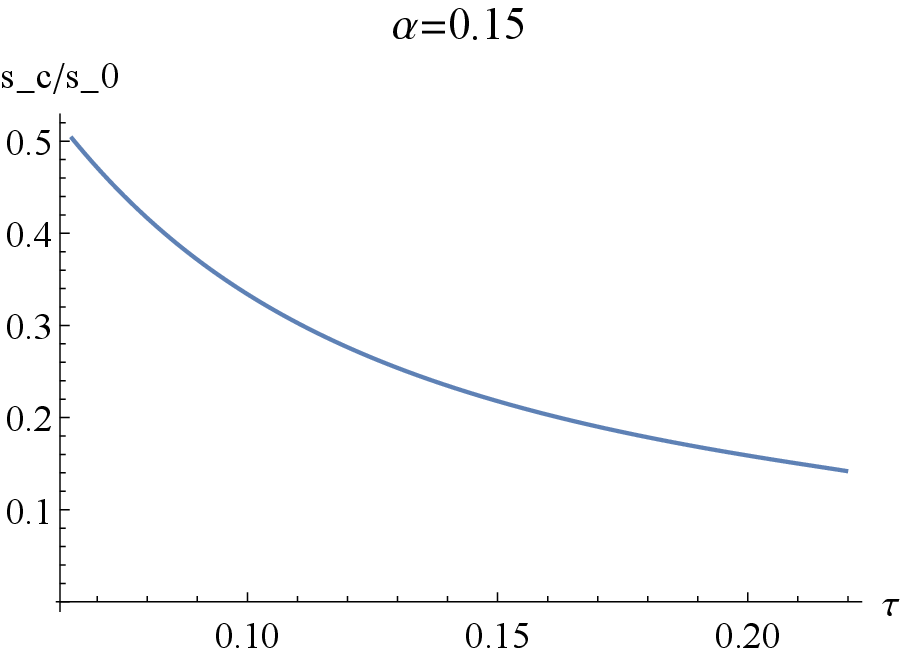}
\caption{$\alpha=0.15$}
\end{subfigure}
\begin{subfigure}[b]{0.3\textwidth}
\includegraphics[width=0.9\textwidth]{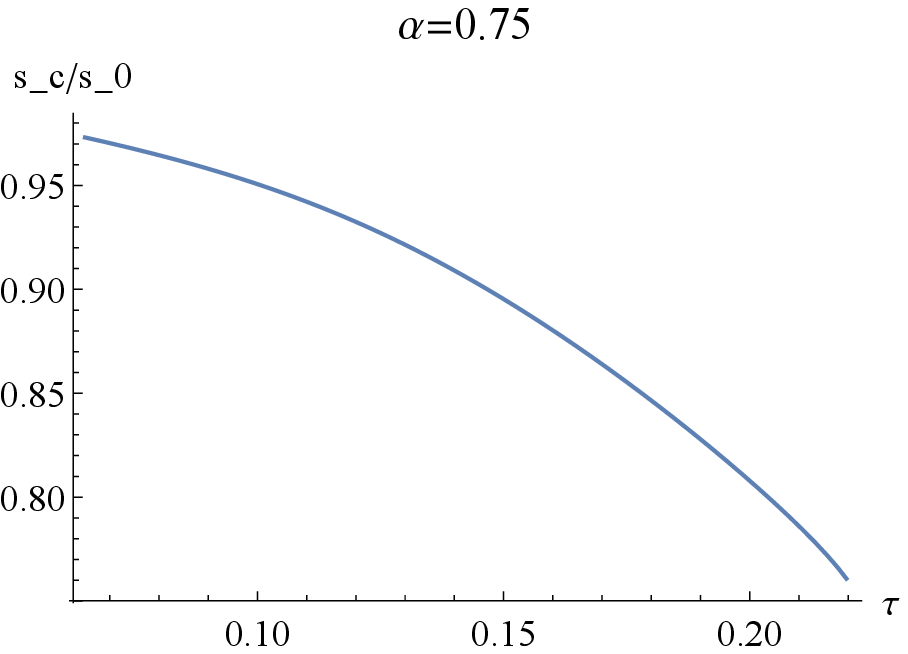}
\caption{$\alpha=0.75$}
\end{subfigure}
\begin{subfigure}[b]{0.3\textwidth}
\includegraphics[width=0.9\textwidth]{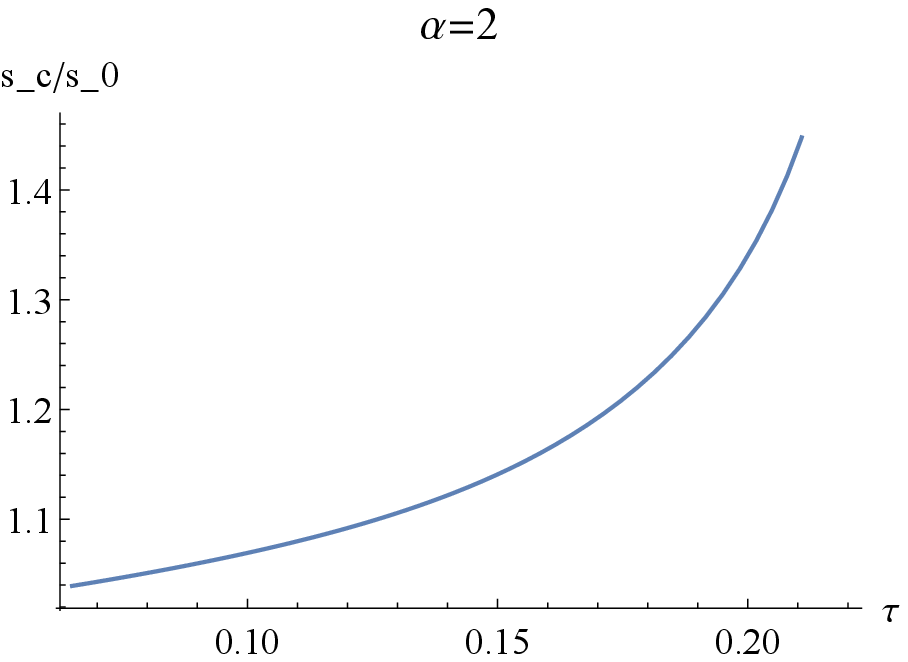}
\caption{$\alpha=2$}
\end{subfigure}
\caption{Plots of $\tilde{h}$ as a function of $\tau$, for representative values of $\alpha$.}
\end{figure}

We also include a contour plot showing a range of $(\tau,\alpha)$ space and the corresponding values of $\tilde{h}$. \begin{figure}\begin{center}
\includegraphics[width=0.7\textwidth]{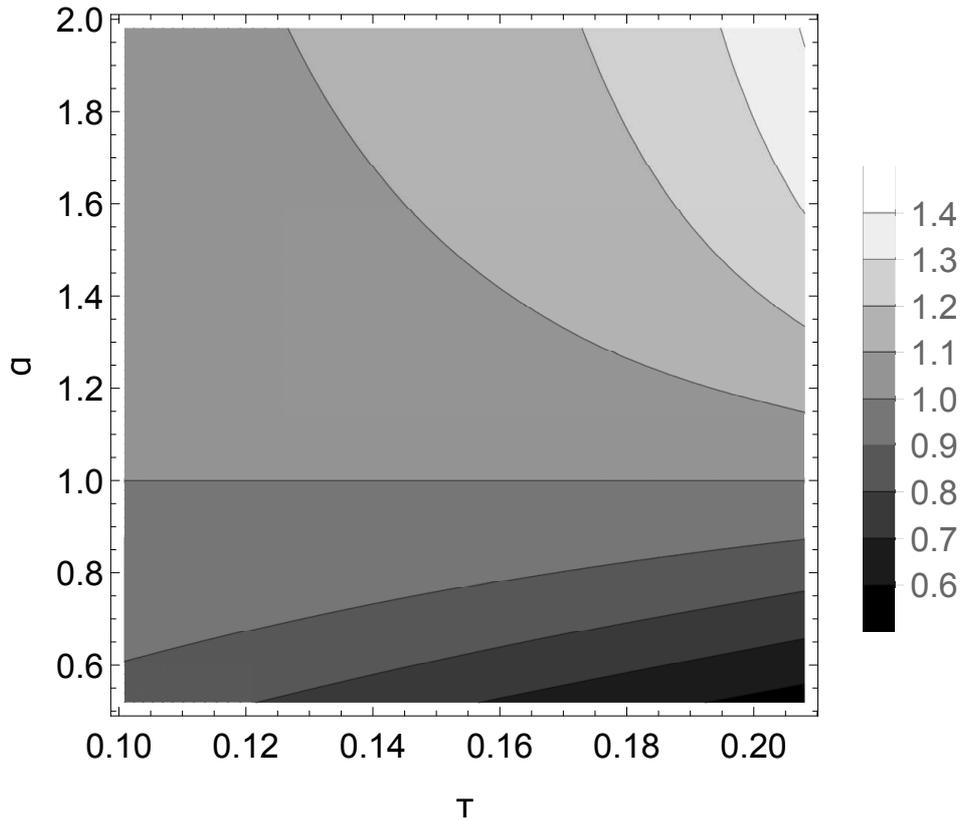}\end{center}\caption{A contour plot of the dimensionless HTP $\tilde{h}$ plotted against parameters $\alpha$,$\tau$.}\label{figHTP}\end{figure}

\section{Conclusions}
Within this work we have proposed a mean-field model for a spatially inhomogeneous two-species Host-Dopant mixture based on simple pairwise interactions. Both formally and in the language of $\Gamma$-convergence, we have analysed the asymptotics of this model in a large-domain and dilute-dopant limit, where the dopant number density scales as the reciprocal of domain size. This gives the Oseen-Frank free energy for cholesterics, where the Frank constants and ground state wavenumber admit explicit representations in terms of the molecular interactions and equilibrium order parameters of the two species. In particular, this gives an expression for the HTP, which may be increasing, decreasing or independent with respect to temperature depending on the relative strength of the Host-Host and Host-Dopant interactions. 

\section{Acknowledgments}
The author would like to express gratitude to Tianyi Guo, Peter Palffy-Muhoray, and Xiaoyu Zheng, for many insightful discussions in the making of this work. This research has been partially supported by the Basque Government through the BERC 2018-2021 program; and by Spanish Ministry of Economy and Competitiveness MINECO through BCAM Severo Ochoa excellence accreditation SEV-2017-0718 and through project MTM2017-82184-R funded by (AEI/FEDER, UE) and acronym ``DESFLU"

\bibliography{bibl}

\appendix

\section{Symmetries}\label{appSymmetry}

\subsection{General invariants for the antisymmetric interaction}
In this section we discuss the general form of an antisymmetric interaction kernel. We assume that $\tilde{\mathcal{K}}^c_{XY}(z,p,q)$ is linear in $\sigma(p)$ and $\sigma(q)$. By viewing $\tilde{\mathcal{K}}_{XY}^c(z,p,q)=A(z,\sigma(p))\cdot \sigma(q)$, where $A$ defines a traceless symmetric matrix, we can use existing results on invariants of hemitropic tensor valued functions to obtain a general form (see e.g. \cite[Section 4.4]{liu2013continuum}). Given a vector $z$, define $W_z$ to be the skew-symmetric matrix so that $W_zx=z\times x$ for all $x$. Then, we must have that for $W=W_z$, $Q=\sigma(p)$, as $A$ is hemitropic it must be of the form
\begin{equation}
f_1(z)(QW-WQ)+f_2(z)(WQW^2-W^2QW)+f_3(z)Q+ f_4(z)WQW+f_5(z,Q)z\otimes z.
\end{equation}
Now we consider the terms which are anti-symmetric, having odd symmetry in $z$. This leads us to conclude that $f_3,f_4,f_5$ must have odd symmetry in $z$, while $f_1,f_2$ must have even symmetry in $z$. As there are no isotropic functions of a single vector with odd symmetry apart from $0$, this implies $f_3=f_4=0$. By similar symmetry arguments, we must have that 
\begin{equation}
f_4(z,Q)=\tilde f_4(|z|,\text{tr}Q^2, \text{tr}Q^3,z\cdot Qz,z\cdot Q^2z, z\cdot Qz\times Q^2z).
\end{equation}
However, as $f_4$ is required to be linear in $Q$ and have odd symmetry in $z$, this must be of the form $g_4(z)z\cdot Qz$ for odd, isotropic $g_4$, therefore $g_4$ must be zero also and $f_4=0$. This leads us to the conclusion that the interaction must be of the form 
\begin{equation}
\tilde{\mathcal{K}}_{XY}^c(z,p,q)=f_1(z)(\sigma_1W-W\sigma_1)\cdot\sigma_2+f_2(z)(W\sigma_1W^2-W^2\sigma_1 W)\cdot \sigma_2
\end{equation}
where $\sigma_1=\sigma(p)$, $\sigma_2=\sigma(q)$. It is a tedious but straightforward exercise to verify that 
\begin{equation}
\begin{split}
\text{Tr}(\sigma_1W\sigma_2)=&(z\cdot p\times q)(p\cdot q),\\
\text{Tr}(W\sigma_1W^2\sigma_2)=&(z\cdot p\times q)(p\cdot z)(q\cdot z)-|z|^2(z\cdot p\times q)(p\cdot q).
\end{split}
\end{equation}

So, concluding, we claim that if we have an antisymmetric interaction $\tilde{\mathcal{K}}_{XY}^c(z,p,q)$ which is linear in $\sigma(p),\sigma(q)$, then it must be of the form 
\begin{equation}
\tilde{\mathcal{K}}_{XY}^c(z,p,q)=F_1(|z|)(z\cdot p\times q)(p\cdot q)+F_2(|z|)(z\cdot p \times q)(z\cdot p)(z\cdot q).
\end{equation}

At other points it may be more useful to consider a full tensorial form. If $Q^1,Q^2$ are symmetric, traceless tensors, then we can write 
\begin{equation}\begin{split}
\text{Tr}(Q^1WQ^2)=&Q^1_{ij}\epsilon_{jkl}z_kQ^2_{li},\\
\text{Tr}(WQ^1W^2Q^2)=&|z|^2\epsilon_{ijk}z_jQ^1_{kl}(\hat z_l\hat z_m-\delta_{lm})Q^2_{mi}\end{split}
\end{equation}

\end{document}